\documentclass[article]{article}
\usepackage{color}
\usepackage{graphicx}
\usepackage{amsmath}
\usepackage{amssymb}
\usepackage{mathrsfs}
\usepackage[numbers,sort&compress]{natbib}
\usepackage{amsthm}
\numberwithin{equation}{section}
\usepackage{pgf}
\usepackage{CJK}
\usepackage{graphicx}
\usepackage{tikz}
\usepackage{stmaryrd}
\usepackage{graphicx}
\usepackage{hyperref}
\usepackage[latin1]{inputenc}
\usepackage[T1]{fontenc}
\usepackage[english]{babel}
\usepackage{listings}
\usepackage{xcolor,mathrsfs,url}
\usepackage{amssymb}
\usepackage{amsmath}
\usepackage{ifthen}
\newtheorem{theorem}{Theorem}
\newtheorem{lemma}{Lemma}
\newtheorem{remark}{Remark}
\newtheorem{corollary}{Corollary}
\newtheorem{proposition}{Proposition}

\newtheorem{Assumption}{Assumption}
\usepackage {caption}
\usepackage{float}
\usepackage{subfigure}

\DeclareMathOperator*{\res}{Res}

\begin{document}
     \title{ Long time asymptotics for the focusing nonlinear Schr\"odinger equation in   the solitonic region  with the presence of high-order discrete spectrum }
\author{Zhaoyu Wang$^1$, Meisen Chen$^1$ and Engui FAN$^{1}$\thanks{\ Corresponding author and email address: faneg@fudan.edu.cn } }
\footnotetext[1]{ \  School of Mathematical Sciences  and Key Laboratory of Mathematics for Nonlinear Science, Fudan University, Shanghai 200433, P.R. China.}

\date{ }

\maketitle

\begin{abstract}
	\baselineskip=17pt
    In this paper,  we use the $\bar{\partial}$ steepest descent method to study the initial value problem for focusing nonlinear Schr\"odinger (fNLS) equation with non-generic weighted Sobolev initial data that    allows for the presence of  high-order discrete spectrum.
    More precisely, we shall  characterize the properties  of the eigenfunctions and  scattering coefficients in the presence of  high-order poles;
    further we formulate an appropriate enlarged RH problem;
    after a series of deformations, the RH problem is transformed into a solvable model.
    Finally, we obtain the  asymptotic expansion of the solution of the fNLS equation in any fixed space-time cone:
    \begin{equation*}
        \mathcal{S}(x_1,x_2,v_1,v_2):=\left\lbrace (x,t)\in \mathbb{R}^2: x=x_0+vt, \ x_0\in[x_1,x_2]\text{, }v\in[v_1,v_2]\right\rbrace.
     \end{equation*}
    Observing the result indicates that the solution of fNLS equation in this case satisfies the soliton resolution conjecture.  The leading order term of this solution includes a
     high-order pole-soliton whose parameters are affected by soliton-soliton interactions through the cone and soliton-radiation interactions on continuous spectrum.
   The error term of this result is up to $\mathcal{O}(t^{-3/4})$ which comes from the corresponding $\bar{\partial}$ equation.
\end{abstract}

\baselineskip=17pt

\newpage

\tableofcontents

\section {Introduction}

In this paper, we study the long time asymptotic behavior of  the Cauchy problem for the focusing nonlinear Schr\"odinger (fNLS) equation in   the solitonic region  with high-order discrete spectrum
   \begin{align}
    &i q_{t}+\frac{1}{2}q_{xx}+\left| q \right|^2 q=0,\label{fNLS}\\
    &q(x,0)=q_{0}(x) \in  H^{1,1}(\mathbb{R})\label{fNLSid},
   \end{align}
where $ H^{1,1}(\mathbb{R})$  is a weighted Sobolev space
    \begin{equation}
        H^{1,1}(\mathbb{R})=\left\{f(x)\in L^{2}(\mathbb{R}) :f'(x),xf(x)\in L^{2}(\mathbb{R})\right\}.
    \end{equation}
The NLS equation is an important model in applied mathematics and theoretical
physics due to both its surprisingly rich mathematical structure and its physical
significance and broad applicability to a number of different areas \cite{NFO}-\cite{BEC}.
The NLS equation is also  a completely integrable system.  A Lax
pair for the equation was first  derived by Zakharov and Shabat  in 1972 \cite{ZS1}.
For sufficient smoothness of the initial data,   Zakharov and Shabat   developed   the inverse
scattering transform (IST)  for the initial value problem   of the NLS
on the line for initial conditions with sufficiently rapid decay at infinity \cite{ZS1}.
One year later, the  IST for the defocusing NLS equation on the line with
nonzero boundary conditions (NZBC)  at infinity was developed   \cite{ZS2}.
 The periodic problem for NLS was studied by Its and Kotlyarov in 1976 \cite{Its1}.
Biondini and Kovacic established asymptotic expressions for the multiple pole solutions of the fNLS equation via the IST method \cite{MPS}.
The  N-soliton  solutions   for the fNLS equation with NZBC at infinity and double zeros of the analytic scattering coefficients has been studied by
 Pichler and Biondini \cite{fNLSdps}.  The  well-posedness of the NLS equation on the line with
initial data in  $L^2$  and in  Sobolev spaces $ H^s, \ s>0 $  was proved  by  Tsutsumi  and Bourgain  respectively \cite{Tsutsumi,Bourgain}.

The long time asymptotic behavior of the defocusing  NLS equation  with Schwartz initial data  was first studied by  Zakharov and Manakov   by the IST method \cite{NWS}.
 The  focusing NLS equation with nonzero boundary conditions by the IST
method were presented by Kawata, Inoue and Ma in \cite{Kawata, Ma}. Using monodromy theory, Its was able to reduce the RH  problem formulation for the  NLS  equation
to a model case, which can then be solved explicitly, giving the desired asymptotics \cite{Its}. A perturbation theory for the NLS equation with non-vanishing boundary
conditions was put forward in \cite{Garnier}, where particular attention was paid to the stability
of the Ma soliton.   Whitham theory results for the focusing NLS with step-like data can be found by Bikbaev in  \cite{Bikbaev}.
In particular,  a nonlinear steepest descent method for oscillatory RH problem   was developed  by  Deift and   Zhou in 1993  \cite{sdmRHp}, which is  a new great achievement   in the further development of the IST method.
After that, a   numerous new significant
results on  long-time  asymptotics   for NLS equation also other integrable nonlinear equations
 have been obtained  in a rigorous and transparent form with this new method  \cite{NFNLS,dNLS,SPE}.
Kamvissis obtained the long time behavior for the focusing  NLS  equation with real spectral singularities \cite{Kamvissis1}.
Boutet de Monvel et al.  studied   long time   asymptotic behavior   of  the fNLS equation    with time-periodic boundary
condition on the half-line \cite{Monvel3},   with  step-like  initial data \cite{Monvel4}, and more  general step-like  initial data  recently \cite{Monvel5}.
By using a variant of  IST   and by employing   Deift-Zhou nonlinear steepest descent
method, Biondini studied  the long time asymptotic behavior of the focusing  NLS  equation on the line with symmetric, nonzero boundary conditions
at infinity \cite{Biondini1}, and  recently with nonzero boundary conditions
in the presence of a discrete spectrum \cite{Biondini2}.

Most recently, for   weighted Sobolev initial data  $q_0(x)\in H^{1,1}(\mathbb{R})$,   Borghese et al  applied the $\bar{\partial}$  steepest descent method   to
obtain   asymptotic expansion in   any fixed space-time cone   for the focusing NLS   equation in  solitonic region  \cite{fNLS};  the $\bar{\partial}$  steepest descent method
was first applied  to analyze  the asymptotics  of orthogonal polynomials on the unit circle and  on real line   by McLaughlin and Miller in 2006 \cite{OPsuc,OPsrlvw}.
Later, this method was further generalized  to widely study the  long time asymptotics  of  integrable systems.
For example, Cuccagna and Jenkins studied  the large-time leading order approximation  and  the asymptotic stability of $N$-soliton solutions of the defocusing NLS equation in 2016 \cite{Asystab}.
Jenkins et al. obtained the soliton resolution property of the derivative NLS equation: As $t$ approaches infinity, the solutions can be described by a finite sum of localized solitons and a dispersive component \cite{srfdnsl}.  We recently obtained   long time asymptotics  of  short pulse equation in solitonic region \cite{Yang1}.   The advantages of this method are not  only  avoiding  delicate estimates   of Cauchy projection operators but also  improving  error estimates without additional restrictions on the initial data.

   For the defocusing NLS equation,   its   ZS-AKNS operator is  self adjoint, so no soliton solutions  appear due to  empty  discrete spectrum
 for finite mass initial data $q_0(x)\in H^{1,1}(\mathbb{R})$.
Soliton solutions have no effect on the long-time asymptotic behavior.
 However, for the  focusing case,
the ZS-AKNS operator   is non-self adjoint that allow for  presence of solitons  anywhere in $ \mathbb{ C} \setminus \mathbb{R}$.    It  is necessary to consider  effects
 of soliton solutions when we study long time asymptotic behavior.
 Therefore, the long-time behavior of solutions of fNLS are necessarily more detailed than in the defocusing case due to the presence of solitons
   which correspond to discrete spectrum of the non self-adjoint ZS-AKNS  scattering operator.
The corresponding reflection coefficient $r(z)$  is
 a mapping defined on the real axis $r(z): \mathbb{R} \rightarrow \mathbb{C}$.
It is possible for $r(z)$ to possess singularities along the real line and we  call these points  spectral singularities.
The initial data $q_{0}$, which has no spectral singularities and produces only simple discrete spectrum, is generic.
If spectral singularities or high-order discrete spectrum exist,  the initial data $q_{0}$ is called  the non-generic.

For the focusing NLS equation with
zero boundary conditions (ZBC), it has been known that more general solutions
corresponding to double order  poles exist since the original work of Zakharov and  Shabat \cite{ZS1}.
More general  high-order pole  solutions of the focusing NLS equation with ZBC were also studied by  Aktosun et al \cite{Aktosun,Schiebold}.
Such solutions also exist in the focusing case with NZBC and  describe the interaction of two
solitons with same amplitude and velocity parameters, which diverge from each other logarithmically  as
in the case of zero boundary conditions \cite{fNLSdps}.  Indeed for focusing NLS equation, it can be shown that  high order discrete spectrum may appear
see in the following  Lemma \ref{lemma1}.
 As is common,  however,  all long time asymptotic  expressions
 of the focusing NLS equation in the solitonic region   are limited to
  the case in which all discrete spectrum  are simple  \cite{Monvel5,Biondini1,Biondini2,fNLS}.
    To the best of our knowledge,
 none of those works  studied   long time asymptotic  expressions  of the focusing NLS equation   with  high-order discrete spectrum   in the framework of the RH method.
  A natural question is therefore  how to  characterize the long time behavior of the focusing NLS equation in  the presence of high-order discrete spectrum.

 In this work, we provide an implementation of the above  question.  We apply the $\bar{\partial}$ steepest descent techniques   to obtain the long-time asymptotic behavior of solutions
  for the Cauchy problem (\ref{fNLS})-(\ref{fNLSid}) of the   fNLS equation  with   non-generic initial data which allows for the presence of high-order discrete spectrum.
 More precisely, we shall  characterize the properties  of the eigenfunctions and  scattering coefficients in the presence of  high-order poles;
further we formulate an appropriate enlarged RH problem;  after a series of deformations, the RH problem
is transformed into a solvable model.
Finally, we  obtain the long time asymptotic expression  of the focusing NLS equation in
solitonic  region  with  the presence of high-order discrete spectrum.

The structure of this work is the following:  In Section 2, we recall the basic scattering theory about the fNLS equation, such as the Lax pair, the analyticity and  the symmetry of the corresponding eigenfunctions.
In Section 3,  we consider the high-order discrete spectrum and compute the residue condition and the coefficients of negative power terms.
In Section 4,  we formulate an RH problem  $m(z)$ to characterize the Cauchy problem (\ref{fNLS})-(\ref{fNLSid})   with   high-order poles.
In Section 5, in order to regularize the RH problem $m(z)$, we first study the property of the jump matrices and introduce a transformation $T(z)$ to get $m^{(1)}(z)$.
We then make continuous extension of these jump matrices  to obtain a mixed $\bar{\partial}$-RH problem $m^{(2)}(z)$.
In Section 6, we decompose  $m^{(2)}(z)$ into a pure RH problem $m^{(2)}_{RHP}(z)$ and a pure $\bar{\partial}$-problem $m^{(3)}(z)$,
 while   the RH problems about   $m^{(2)}_{RHP}(z)$ and  $m^{(3)}(z)$ can be shown solved  respectively.
Finally,  we give an explicit formula for the solution of the fNLS equation in Section 7.
Moreover, the property of soliton resolution can be obtained after analyzing the  form of solution.

\section {The Lax pair and   spectral analysis}

The fNLS equation (\ref{fNLS}) admits the Lax pair \cite{fNLS}
  \begin{equation}\label{laxp}
    \Phi_x+iz \sigma_3 \Phi=Q_1 \Phi,\quad \Phi_t+iz^2 \sigma_3 \Phi=Q_2 \Phi,
    \end{equation}
 where
 \begin{equation}
    \sigma_3=\left(\begin{array}{cc}
    1& 0  \\
    0 & -1
    \end{array}\right),\hspace{0.5cm} Q_1=\left(\begin{array}{cc}
    0 & q \\
    -\bar{q}& 0
    \end{array}\right),\hspace{0.5cm} Q_2=\frac{1}{2}\left(\begin{array}{cc}
    i\left| q \right|^2  & iq_x+2zq   \\
    i\bar{q}_x-2z \bar{q}& -i\left| q \right|^2
    \end{array}\right).\hspace{0.5cm} \nonumber
    \end{equation}
  Given the initial condition (\ref{fNLSid}), the Lax pair (\ref{laxp}) has a solution of the following asymptotic form
  \begin{equation}
    \Phi(z) \sim e^{-i(zx+z^2t)\sigma_3}, \hspace{0.5cm}x\to \pm\infty.\label{asyx}
  \end{equation}
  By making a transformation
  \begin{equation}
    \mu(z)=\Phi(z) e^{i(zx+z^2t)\sigma_3},\label{trans2}
  \end{equation}
  we find the matrix function $\mu$ has the following asymptotic behavior
   \begin{equation*}
    \mu(z) \sim I, \hspace{0.5cm} x \rightarrow \pm\infty
   \end{equation*}
   and satisfies the following Lax pair
   \begin{align}
    &\mu_x + iz[\sigma_3,\mu]=Q_1\mu,\label{lax0.1}\\
    &\mu_t +iz^2[\sigma_3,\mu]=Q_2\mu. \label{lax0.2}
    \end{align}
   This Lax pair (\ref{lax0.1})-(\ref{lax0.2}) can be written by fully differential form
   \begin{equation}
    d\left(e^{i(zx+z^2t)\hat{\sigma}_3}\mu\right)=e^{i(zx+z^2t)\hat{\sigma}_3}[(Q_1 dx+Q_2 dt)\mu].
   \end{equation}
   We expand  $\mu$ into a Taylor series at infinity and prove that
   \begin{align}
    &\mu(z) \sim  I,\hspace{0.5cm} z \rightarrow \infty,  \label{lax0.3}\\
    &q(x,t)=2i \lim_{z\to \infty}(z\mu)_{12}. \label{lax0.4}
    \end{align}
    By integrating the equation in two directions parallel to the real axis, two eigenvalue functions can be obtained
    \begin{align}
        &\mu^-(z;x,t)=I+ \int^{x}_{-\infty}e^{-iz(x-y)\hat{\sigma}_3}Q_1(z;y,t)\mu^-(z;y,t)dy, \label{intmu0}\\
        &\mu^+(z;x,t)=I-\int^{+\infty}_{x}e^{-iz(x-y)\hat{\sigma}_3}Q_1(z;y,t)\mu^+(z;y,t)dy. \label{intmu1}
    \end{align}

    From the relation (\ref{trans2}), we know that
    \begin{equation}
        \Phi^{\pm}(z) = \mu^{\pm}(z) e^{-i(zx+z^2t)\sigma_3}\label{trans3}
      \end{equation}
    are two linear correlation matrix solutions of the Lax pair (\ref{laxp}), which means there is a matrix $S(z)=\left( s_{ik}(z) \right)^2_{i,k=1}$ satisfying the condition
    \begin{equation}
        \Phi^{-}(z)=\Phi^{+}(z)S(z).
    \end{equation}
    Therefore, we obtain
    \begin{equation}\label{Rus}
        \mu^{-}(z)=\mu^{+}(z) e^{-i(zx+z^2t)\sigma_3}S(z),
    \end{equation}
    where the matrix function $S(z)$ is called the spectral matrix and $s_{ik}(z),i,k=1,2$ is called the scattering data.
    Through direct calculations, we get
    \begin{align}
       & s_{11}(z)= \text{det} \left( \mu_1^-, \mu_2^+ \right) = 1+  \int^{\infty}_{-\infty} \bar{q}(y)  \mu^+_{12}(y)dy=1+ \int^{\infty}_{-\infty} q(y)  \mu^-_{21}(y)dy, \label{s11}\\
       & s_{21}(z)= \text{det} \left( \mu_1^+, \mu_1^- \right) = - \int^{\infty}_{-\infty} \bar{q}(y) e^{-2izy}  \mu^+_{11}(y)dy=-\int^{\infty}_{-\infty} q(y)e^{-2izy}  \mu^-_{22}(y)dy,
    \end{align}
    where we denote
    \begin{equation*}
        \mu^{\pm}=( \mu^{\pm}_{1},\mu^{\pm}_{2})= \left( \begin{array}{cc}
            \mu^{\pm}_{11}& \mu^{\pm}_{12} \\
            \mu^{\pm}_{21}& \mu^{\pm}_{22}
            \end{array}\right).
    \end{equation*}
    When $q(x)\in L^1(\mathbb{R})$, by constructing iterative sequence and Neumann series, we can prove that $\mu^-_{1}$, $\mu^+_{2}$, $s_{11}$  are analytic in the upper half complex plane;
    $\mu^{-}_{2}$, $\mu^+_{1}$, $s_{22}$  are analytic in the lower half complex plane;
    $s_{12}$ and $s_{21}$ are not analytic in the upper and lower half complex plane but are continuous on the real axis.

    In addition, we can find  symmetries of $\mu^{\pm}$ and $S(z)$
    \begin{equation}
        \mu^{\pm}(z)=- \sigma  \overline{ \mu^{\pm}(\bar{z})  } \sigma =  \sigma_2  \overline{\mu^{\pm}(\bar{z})} \sigma_2,
    \end{equation}
    \begin{equation}
        S(z)=- \sigma  \overline{ S(\bar{z})  } \sigma =  \sigma_2  \overline{S(\bar{z})} \sigma_2,
    \end{equation}
    where
    \begin{equation*}
        \sigma =\left( \begin{array}{cc}
            0& 1  \\
            -1 & 0
            \end{array}\right),\hspace{0.5cm} \sigma_2=\left( \begin{array}{cc}
                0& -i  \\
                i & 0
                \end{array}\right).
    \end{equation*}

    Here we give the definitions of several important concepts: the reflection coefficient $r(z)=s_{21}(z)/s_{11}(z)$ and the transmission coefficient $\tau(z)=1/s_{11}(z)$.
    In particular, for $z\in\mathbb{R}$, we have $s_{11}(z)=\overline{s_{22}(z)} $, $s_{12}(z)=-\overline{s_{21}(z)} $, and $1+\left| r(z) \right|^2=\left| \tau(z) \right|^2$.

    For  simplicity, we give an assumption about the initial data and  scattering data.
    \begin{Assumption}\label{initialdata}
        The initial data $q_0 \in H^{1,1}(\mathbb{R})$ and the corresponding scattering data satisfy the following conditions:
        $s_{11}(z)$ has no zeros on $\mathbb{R}$;  $s_{11}(z)$ only  has finite double  roots; $s_{11}(z)$, $r(z)\in H^{1,1}(\mathbb{R})$.

    \end{Assumption}

    We give the following lemma to illustrate the rationality of our assumption.
    \begin{lemma} \label{lemma1}
       The zeros of $s_{11}(z)$ in $\mathbb{C}^+$ are finite but not necessarily simple in the case  $q_0 \in H^{1,1}(\mathbb{R})$.
    \end{lemma}
    \begin{proof}
       For $\Phi^\pm(z)=\left(\Phi^\pm_1(z), \Phi^\pm_2(z) \right)$, applying (\ref{trans3}) to (\ref{s11}) gives
        \begin{equation}\label{s112}
            s_{11}(z)=\text{det} \left(  \Phi_1^-(z),  \Phi_2^+(z) \right).
        \end{equation}
        Suppose that $z_k \in \mathbb{C}^+$$\left(k=1,...,N\right)$ are the zeros of $s_{11}(z)$.
        From (\ref{s112}), we know the pair $\Phi_1^-(z_k)$ and $\Phi_2^+(z_k)$ are linearly related, which is there exists a constant $\gamma_k \in \mathbb{C}$ such that
        \begin{equation}
            \Phi_1^-(z_k)= \gamma_k \Phi_2^+(z_k).
        \end{equation}

        Then, we consider the partial derivative of $s_{11}(z)$
        \begin{equation}
            \frac{\partial s_{11}(z) }{ \partial z} \bigg|_{z=z_k}= \text{det} \left( \partial_z \Phi_1^-, \Phi_2^+  \right) + \text{det} \left( \Phi_1^-, \partial_z \Phi_2^+  \right)  \bigg|_{z=z_k}.
        \end{equation}
        Using (\ref{laxp}), we find that
        \begin{align}
           & \frac{\partial}{\partial x} \text{det} \left( \partial_z \Phi_1^-, \Phi_2^+  \right)= -i \; \text{det} \left( \sigma_3 \Phi_1^-, \Phi_2^+  \right), \\
           & \frac{\partial}{\partial x} \text{det} \left( \Phi_1^-, \partial_z \Phi_2^+  \right)= -i \; \text{det} \left( \Phi_1^-, \sigma_3 \Phi_2^+  \right).
        \end{align}
       The equations (\ref{intmu0}), (\ref{intmu0}) and (\ref{trans3}) tell us that
        \begin{align}
             &\Phi_1^-(z;x) \sim \left( \begin{array}{ll}
                1\\
                0
                \end{array}\right) e^{-2it \theta(z;x)}, \quad x \to -\infty,\\
            &\Phi_2^+(z;x) \sim \left( \begin{array}{ll}
                0\\
                1
                    \end{array}\right) e^{2it \theta(z;x)}, \quad x \to +\infty,\\
            & \partial_z \Phi_1^-(z;x) \sim \left( \begin{array}{ll}
                - i \left( x+2zt  \right) \\
                0
                \end{array}\right) e^{-2it \theta(z;x)}, \quad x \to -\infty.
        \end{align}
       Then, we obtain
       \begin{align}
        &\text{det} \left( \partial_z \Phi_1^-, \Phi_2^+ \right)= - i \gamma_k \int_{-\infty}^{x} \text{det} \left( \sigma_3 \Phi_2^+(z_k;s),  \Phi_2^+(z_k;s)     \right) ds,\\
        &\text{det} \left(  \Phi_1^-, \partial_z \Phi_2^+ \right)= - i \gamma_k \int^{\infty}_{x} \text{det} \left( \sigma_3 \Phi_2^+(z_k;s),  \Phi_2^+(z_k;s)     \right) ds.
      \end{align}
      Putting the above two terms together,  we have
      \begin{equation}
        \frac{\partial s_{11}(z) }{ \partial z} \bigg|_{z=z_k}= - 2 i \gamma_k \int_{-\infty}^{\infty}  \Phi_{12}^+(z_k;s) \Phi_{22}^+(z_k;s) ds.
      \end{equation}
      Therefore, we can find: When the condition $\int_{-\infty}^{\infty}  \Phi_{12}^+(z_k;s) \Phi_{22}^+(z_k;s) ds=0$ is satisfied, the zero $z_k$ is not simple.
      That means $z_k$ might be a multiple zero of $s_{11}(z)$.

      To prove that the number of zeros of $s_{11}$ is finite, we first suppose that $s_{11}(z)$ has no zeros on $\mathbb{R}$.
      Using the asymptotic behavior of $s_{11}(z)$ ($s_{11}(z) \to 1$ as $z \to \infty$), we can  give the finiteness of the number of zeros of $s_{11}$.

    \end{proof}

\section{Discrete spectrum with double poles}
    Now we suppose that $s_{11}(z)$ has $N$ double zeros, which is $s_{11}(z_k)= s'_{11}(z_k)=0$ and $s''_{11}(z_k) \ne 0$, in the upper half complex plane $\mathbb{C}^+$ and denote them by $z_k$$\left(k=1,...,N\right)$.
    By the symmetry of the eigenfunction, we know that $\bar{z}_k \in \mathbb{C}^- \left(k=1,...,N\right)$ are the double zeros of $s_{22}(z)$.
    Denote
    \begin{align}
    &\mathcal{Z}=\left\{ z_k | s_{11}(z_k)= s'_{11}(z_k)=0, s''_{11}(z_k) \ne 0\right\},\nonumber\\
    &\bar{\mathcal{Z}}=\left\{ \bar{z}_k | s_{22}(\bar{z}_k)= s'_{22}(\bar{z}_k)=0, s''_{22}(\bar{z}_k) \ne 0\right\}, \nonumber
    \end{align}
 which are the sets of the zeros of $s_{11}(z)$ and $s_{22}(z)$ respectively.

   From the relation (\ref{Rus}) and $s_{11}(z_k)=s'_{11}(z_k)=0$, we deduce that there are norming constants $b_k$ and $d_k$ that are independent of $x$ and $t$ such that
   \begin{equation} \label{u01-}
   \mu^-_{1}(z_k)=b_k e^{2it \theta(z_k)}\mu^+_{2}(z_k),
   \end{equation}
   \begin{equation} \label{u101-}
    \left(\mu^{-}_{1}\right)'(z_k)= e^{2it \theta(z_k)} \left( \left(2it \theta'(z_k)b_k+d_k\right)   \mu^+_{2}(z_k)        +b_k\left(\mu^+_{2}\right)' \left(z_k\right)  \right).
    \end{equation}
    where $\theta(z)=z^2+xz/t$.
   Similarly,
   \begin{equation} \label{u02-}
    \mu^-_{2}(\bar{z}_k) =\hat{b}_k \theta  (\bar{z}_k) e^{-2it \theta (\bar{z}_k) }\mu^+_{1}(\bar{z}_k)
   \end{equation}
    \begin{equation} \label{u102-}
     \left(\mu^{-}_{2}\right)'(\bar{z}_k)= e^{-2it \theta(\bar{z}_k)} \left( \left( \hat{d}_k -2it \theta'(\bar{z}_k) \hat{b}_k\right)   \mu^+_{1}(\bar{z}_k)    +\hat{b}_k \left(\mu^+_{1}\right)'\left(\bar{z}_k\right)  \right).
     \end{equation}
where $\hat{b}_k=-\bar{b}_k$ and $\hat{d}_k=-\bar{d}_k$ according to  the symmetry of $S(z)$.

    Notice that $\mu^{-}_{1}$ is analytic in the upper half plane $\mathbb{C}^+$ and $z_k$ is the double zero of $s_{11}(z)$, then let $\mu^{-}_{1}$ and $s_{11}(z)$ do Taylor expansion at point $z_k$
    \begin{align} \label{Taylor1}
        \frac{\mu^{-}_{1}(z)}{s_{11}(z)}&=\frac{\mu^{-}_{1}(z_k)+\left(\mu^{-}_{1}\right)'(z_k)(z-z_k)+\left(\mu^{-}_{1}\right)''(z_k)(z-z_k)^2/2+...}{s''_{11}(z_k)(z-z_k)^2/2+s'''_{11}(z_k)(z-z_k)^3/6+...}\\
       & = \frac{2\mu^{-}_{1}(z_k)}{s''_{11}(z_k)}(z-z_k)^{-2}+ \left(  \frac{2\left(\mu^{-}_{1}\right)'(z_k)}{s''_{11}(z_k)}  -\frac{2\mu^{-}_{1}(z_k)s'''_{11}(z_k) }{3s''_{11}(z_k)^2}  \right)(z-z_k)^{-1}+...
    \end{align}
    The above equations (\ref{u01-}), (\ref{u101-}) and (\ref{Taylor1}) yield the residue condition and the  coefficient of $(z-z_k)^{-2}$ in the Laurent  expansion of $ \frac{\mu^{-}_{1}(z)}{s_{11}(z)}$
    \begin{align} \label{Res1}
     \res_{z=z_k} \left[  \frac{\mu^{-}_{1}(z)}{s_{11}(z)}   \right] &=   \frac{2\left(\mu^{-}_{1}\right)'(z_k)}{s''_{11}(z_k)}  -\frac{2\mu^{-}_{1}(z_k)s'''_{11}(z_k) }{3s''_{11}(z_k)^2} \\
     &= A_k e^{2it \theta(z_k)} \left(     \left(\mu^{+}_{2}\right)'(z_k) +   \mu^{+}_{2}(z_k)  \left( B_k+2it \theta'(z_k) \right)     \right),
    \end{align}
    \begin{equation}\label{P-21}
        \underset{z=z_k}{P_{-2}}\left[  \frac{\mu^{-}_{1}(z)}{s_{11}(z)}   \right] = \frac{2\mu^{-}_{1}(z_k)}{s''_{11}(z_k)}=A_k e^{2it \theta(z_k)} \mu^{+}_{2}(z_k),
    \end{equation}
    where
    \begin{equation}\label{AkBk1}
        A_k=\frac{2 b_k}{s''_{11}(z_k)}, \quad B_k=\frac{d_k}{b_k}-\frac{s'''_{11}(z_k)}{3s''_{11}(z_k)}.
    \end{equation}
    Likewise, as $z=\bar{z}_k$ is the double zero of $s_{22}(z)$, by equations (\ref{u02-}), (\ref{u102-}) and (\ref{Taylor1}), we obtain
    \begin{align}  \label{Res2}
        \res_{z=\bar{z}_k} \left[  \frac{\mu^{-}_{2}(z)}{s_{22}(z)}   \right] &=   \frac{2\left(\mu^{-}_{2}\right)'(\bar{z}_k)}{s''_{22}(\bar{z}_k)}  -\frac{2\mu^{-}_{2}(\bar{z}_k)s'''_{22}(\bar{z}_k) }{3s''_{22}(\bar{z}_k)^2} \\
        &= \hat{A}_k e^{2it \theta(\bar{z}_k)} \left(     \left(\mu^{+}_{1}\right)'(\bar{z}_k) +   \mu^{+}_{1}(\bar{z}_k)  \left(\hat{B}_k-2it \theta'(\bar{z}_k)\right)     \right),
       \end{align}
       \begin{equation}\label{P-22}
           \underset{z=\bar{z}_k}{P_{-2}}\left[  \frac{\mu^{-}_{2}(z)}{s_{22}(z)}   \right] = \frac{2\mu^{-}_{2}(\bar{z}_k)}{s''_{22}(\bar{z}_k)}=\hat{A}_k e^{-2it \theta(\bar{z}_k)} \mu^{+}_{1}(\bar{z}_k),
       \end{equation}
       where
       \begin{equation}\label{AkBk2}
           \hat{A}_k=\frac{2 \hat{b}_k}{s''_{22}(\bar{z}_k)}, \quad \hat{B}_k=\frac{\hat{d}_k}{\hat{b}_k}-\frac{s'''_{22}(\bar{z}_k)}{3s''_{22}(\bar{z}_k)}.
       \end{equation}
    Moreover, it is easy to find that
    \begin{equation}
        \hat{A}_k=-\bar{A}_k, \quad \hat{B}_k=\bar{B}_k.
    \end{equation}

\section{The  RH problem   with high-order poles}
In our situation, we introduce the meromorphic matrices
       \begin{equation}
      m(z)=  m(z;x,t)=\left\{ \begin{array}{ll}
        \left( \frac{\mu^{-}_{1}(z)}{s_{11}(z)} ,\mu^{+}_{2}(z) \right),   &\text{as  Im}z>0,\\[12pt]
        \left( \mu^{+}_{1}(z),\frac{\mu^{-}_{2}(z)}{s_{22}(z)}\right)  , &\text{as  Im}z<0,\\
        \end{array}\right.
        \end{equation}
    which satisfies the following RH problem.

    \noindent\textbf{RHP1}.  Find a matrix-valued function $m(z) $ which satisfies
    \begin{itemize}
        \item[(a)]$m(z)$ is meromorphic in $\mathbb{C}\setminus \mathbb{R}$ and has double poles;
        \item[(b)]$m(z)$ satisfies the jump condition $m_+(z)=m_-(z)v(z), \; z \in \mathbb{R}$,
        where
        \begin{equation}\label{V0}
            v(z)=\left(\begin{array}{cc}
                1+|r(z)|^2 & e^{-2it \theta(z)}\overline{r(z)}\\
            e^{2it \theta(z)}r(z) & 1
            \end{array}\right);
            \end{equation}
         \item[(c)]The asymptotic behavior of $m(z)$ at infinity is
       $$m(z)=I+\mathcal{O}(z^{-1}),	\qquad  z \to  \infty; $$
         \item[(d)]$m(z)$  satisfies the residue  and  the coefficient of negative second power term in the Laurent expansion conditions at double zeros $z_k \in \mathcal{Z}$ and $\bar{z}_k \in \bar{\mathcal{Z}}$:
   \begin{align}
   & \res_{z=z_k} m^+(z)=\lim_{z\to z_k} \left[ m'(z)\left(\begin{array}{cc}
                        0 & 0\\
                        A_k e^{2it \theta(z_k)} & 0
                        \end{array}\right)    + m(z) \left(\begin{array}{cc}
                            0 & 0\\
                            A_k \left(B_k+2it \theta'(z_k)\right) e^ {2it \theta(z_k)}& 0
                            \end{array}\right)         \right],\label{res3.3}\\
   & \res_{z=\bar{z}_k} m^-(z)=\lim_{z\to \bar{z}_k} \left[ m'(z)\left(\begin{array}{cc}
                                0 & \hat{A}_k e^{-2it \theta(\bar{z}_k)}\\
                                0 & 0
                                \end{array}\right)    + m(z) \left(\begin{array}{cc}
                                    0 & \hat{A}_k \left(\hat{B}_k-2it \theta'(\bar{z}_k)\right) e^ {-2it \theta(\bar{z}_k)}\\
                                    0& 0
                                    \end{array}\right)         \right],\\
    & \underset{z=z_k}{P_{-2}} m^+(z) = \lim_{z\to z_k} m(z)\left(\begin{array}{cc}
        0 & 0\\
        A_k e^{2it \theta(z_k)} & 0
        \end{array}\right),\\
    &    \underset{z=\bar{z}_k}{P_{-2}} m^-(z) = \lim_{z\to \bar{z}_k} m(z)\left(\begin{array}{cc}
            0 & \hat{A}_k e^{-2it \theta(\bar{z}_k)}\\
           0 & 0
            \end{array}\right).\label{res3.6}
    \end{align}
   \end{itemize}

   \begin{figure}
    \begin{center}
    \begin{tikzpicture}
    \draw[thick,->,red](-3,0)--(3,0);
    \node    at (3.2,0)  {$\mathbb{R}$};
    \node    at (1,1.5)  {$z_k$};
    \node    at (1,-1.5)  {$\bar z_k$};
    \node    at (0.8,1.3)  {$\bullet$};
    \node    at (0.8,-1.3)  {$\bullet$};
    \node    at (2,1 )  {$\bullet$};
    \node    at (2,-1)  {$\bullet$};
    \node    at (-0.8,0.7 )  {$\bullet$};
    \node    at (-0.8,-0.7)  {$\bullet$};
    \node    at (-2,1.4 )  {$\bullet$};
    \node    at (-2,-1.4)  {$\bullet$};
    \node    at (-0.6,1)  {$z_k$};
    \node    at (-0.6,-1)  {$\bar z_k$};
    \end{tikzpicture}
    \end{center}
    \caption{ The  jump contour and  poles for  $m(z)$  }
    \label{fmjump}
    \end{figure}
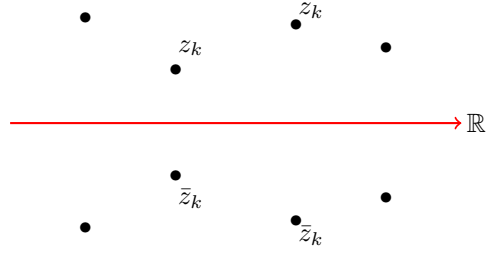

The existence and uniqueness of the above RHP1 can be given by Liouville's theorem and the vanishing lemma \cite{ISTFA}.
  Plugging the asymptotic expansion $m=I+m_1/z+o(z^{-1})$ into the formula (\ref{laxp}), we obtain that
    \begin{equation}\label{mTay}
        m(z)=I +\frac{1}{2iz} \left(\begin{array}{cc}
            -\int^{\infty}_{x} |q|^2 dx & q\\
            \bar{q} & \int^{\infty}_{x} |q|^2 dx
            \end{array}\right)+o(z^{-1}).
    \end{equation}
    Thus, the solution $q(x,t)$ of initial value problem for NLS equation can be expressed by the above RHP1
    \begin{equation}\label{qzm12}
        q(x,t)=2i \lim_{z \to \infty} (zm)_{12}.
    \end{equation}

    \section{Continuous extensions to a mixed $\bar\partial$-RH problem}
    In this section, we make factorizations of the jump matrix $v(z)$ and  continuously extend each factor  off $\mathbb{R}$. The idea of continuous extensions comes from \cite{OPsuc}-\cite{Asystab}.
    Before doing continuous extensions, we renormalize the RH problem of $m(z)$ so that it is well-behaved at infinity.
    Then, we deform the jump matrix onto new contours  on which they decay and obtain a new $\bar\partial$-RH problem by extensions.

    \subsection{Factorizations of jump matrix}
    We first consider the oscillatory term in the jump matrix (\ref{V0})
    \begin{equation*}
        e^{2it \theta(z)}=e^{2t \varphi(z)}, \quad \varphi(z)=i \theta(z)= i(z^2+xz/t).
    \end{equation*}
    Differentiating $\varphi(z)$ with respect to $z$ yields a stationary phase point and four paths of steepest descent
    \begin{equation}\label{Sigma0}
        z_0=-\frac{x}{2t},\quad \Sigma_k= \left\{ z_0+e^{i(2k-1)\pi/4} \mathbb{R}_+     \right\}, k=1,2,3,4.
    \end{equation}
    From $\theta(z)=z^2-2z_0 z=(z-z_0)^2-z_0^2$, we get
    \begin{equation}
     \text{Re} (i\theta)=-2\text{Im}z(\text{Re}z-z_0).
    \end{equation}
    Therefore, we can divide the complex plane into two classes of domains according to the exponential decay $e^{2it \theta(z)}$.

    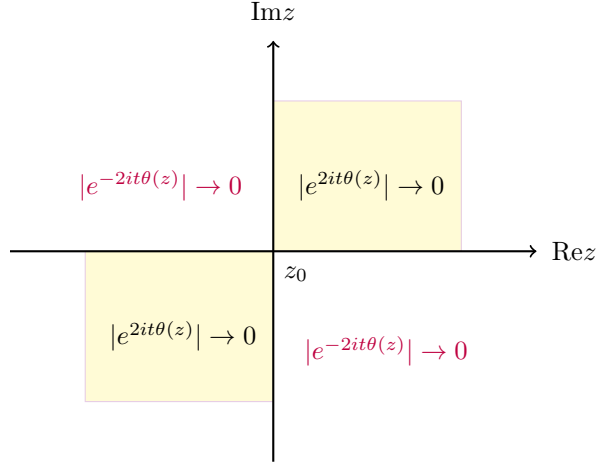
\begin{figure}
        \begin{center}
        \begin{tikzpicture}
        \draw[violet!20, fill=yellow!20] (0,0)--(2.5,0)--(2.5,2)--(0, 2);
        \draw[violet!20, fill=yellow!20] (0,0)--(-2.5,0)--(-2.5,-2)--(0, -2);
        \draw [thick,-> ](-3.5,0)--(3.5,0);
        \draw [thick,-> ](0,-2.8)--(0,2.8);
        \node    at (0.3,-0.3)  {$z_0$};
        \node    at (4,0)  { Re$z$};
        \node    at (0,3.2)  { Im$z$};
        \node  [below]  at (1.3,1.2) {$ |e^{ 2it\theta(z)}| \rightarrow 0$};
        \node  [below,purple]  at (-1.5,1.2) {$|e^{-2it\theta(z)}| \rightarrow 0$};
        \node  [below]  at (-1.2,-0.8) {$|e^{ 2it\theta(z)} |\rightarrow 0$};
        \node  [below,purple]  at (1.5,-1) {$|e^{-  2it\theta(z)}| \rightarrow 0$};
        \end{tikzpicture}
        \end{center}
        \caption{ Exponential   decaying domains  }
        \label{fsigsteep3}
        \end{figure}

    From the above analysis, the jump matrix (\ref{V0}) admits two compositions
    \begin{equation}\nonumber
       v(z) =
        \begin{cases}
        \left(\begin{array}{cc} 1&\bar{r}e^{-2it\theta }\\ 0&1\end{array}  \right)
        \left(\begin{array}{cc} 1&0\\   re^{2it\theta }&1\end{array}  \right), \ \ z>z_0,\\
         \\
         \left(\begin{array}{cc} 1 & 0 \\ \frac{r}{1+|r|^2}e^{2it\theta } & 1\end{array}  \right)
        \left(\begin{array}{cc} 1+|r|^2&0\\ 0&\frac{1}{1+|r|^2} \end{array}  \right) \left(\begin{array}{cc} 1 &
         \frac{\bar{r}}{1+|r|^2}e^{-2it\theta } \\ 0 & 1\end{array}  \right), \ \ z<z_0.
        \end{cases}
        \end{equation}
    To remove the intermediate matrix of the second decomposition, we introduce the following scalar RH problem.

    \noindent\textbf{RHP2}.  Find a scalar function $\delta (z)$ which satisfies
    \begin{itemize}
        \item[(a)]$\delta (z)$ is analytic in $\mathbb{C}\setminus \mathbb{R}$;
        \item[(b)]$\delta_+(z)=\delta_-(z)(1+|r|^2),\quad z<z_0$;
        \item[(c)]$\delta (z) \sim I,\quad  z \to  \infty $.

   \end{itemize}
   By the Plemelj formula,  we prove that this RH problem has a unique solution
    \begin{equation}
    \delta (z)= \exp\left(\frac{1}{2\pi i}\int^{z_0} _{-\infty}\dfrac{\log(1+|r(s)|^2)ds}{s-z}\right) =      \exp\left(i\int^{z_0} _{-\infty}\dfrac{\nu(s)ds}{s-z}\right),
   \end{equation}
   where $\nu(s)= -\frac{1}{2\pi}\log(1+|r(s)|^2)$.

   \subsection{Renormalizations of the RH problem for $m(z)$ }

   For convenience, we introduce some notations
   \begin{align*}
    \Delta^+_{z_0}=\left\lbrace k \in \left\lbrace 1,...,N\right\rbrace  ||z_k|>z_0\right\rbrace, \\
    \Delta^-_{z_0}=\left\lbrace k \in \left\lbrace 1,...,N\right\rbrace  ||z_k|<z_0\right\rbrace.
    \end{align*}
    For a real interval $I=[a,b]$, we define
    \begin{align*}
        &\mathcal{Z}(I)=\left\lbrace z_k \in  \mathcal{Z} ||z_k|\in I \right\rbrace,\\
        &\mathcal{Z}^-(I)=\left\lbrace z_k \in  \mathcal{Z} ||z_k|< a \right\rbrace,\\
        &\mathcal{Z}^+(I)=\left\lbrace z_k \in  \mathcal{Z} ||z_k|> b \right\rbrace.
    \end{align*}
    For a fixed point $z_0 \in I$, we define
    \begin{align*}
        \Delta^-_{z_0}(I)=\left\lbrace k \in \left\lbrace 1,...,N\right\rbrace  |a\le|z_k|<z_0\right\rbrace,\\
        \Delta^+_{z_0}(I)=\left\lbrace k \in \left\lbrace 1,...,N\right\rbrace  |z_0<|z_k|\le b\right\rbrace.
        \end{align*}
    See the corresponding domains for different spectrum sets in Figure \ref{fspectrum}.

        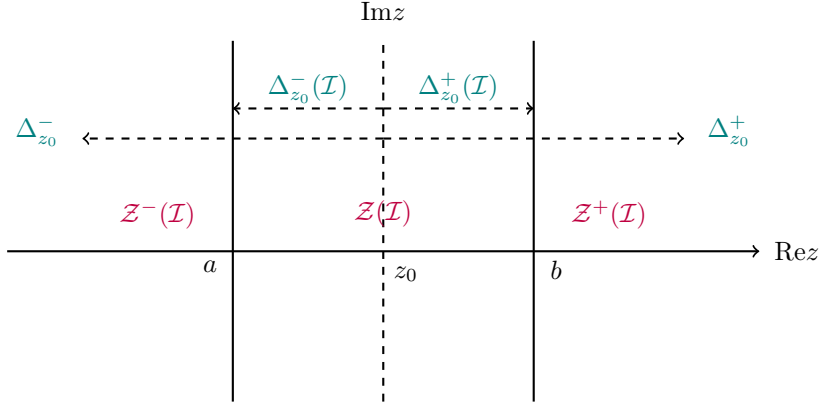
\begin{figure}
            \begin{center}
            \begin{tikzpicture}
            \draw [thick,-> ](-5,0)--(5,0);
            \draw [dashed,thick,-> ](0,1.9)--(2,1.9);
             \draw [dashed,thick,-> ](0,1.5)--(4,1.5);
              \node  [teal]  at (4.6,1.6) {$\Delta_{z_0}^+$};
                \node  [teal]  at (-4.6,1.6) {$\Delta_{z_0}^-$};
            \draw [dashed,thick,-> ](0,1.9)--(-2,1.9);
             \draw [dashed,thick,-> ](0,1.5)--(-4,1.5);
            \draw [dashed,thick ](0,-2)--(0,2.8);
            \draw [thick ](2,-2)--(2,2.8);
            \draw [thick ](-2,-2)--(-2,2.8);
            \node    at (0.3,-0.3)  {$z_0$};
            \node    at (5.5,0)  { Re$z$};
            \node    at (0,3.2)  { Im$z$};
            \node  [below,teal]  at (1,2.5) {$\Delta_{z_0}^+(\mathcal{I})$};
            \node  [below,teal]  at (-1,2.5) {$\Delta_{z_0}^-(\mathcal{I})$};
            \node  [below,purple]  at (-3,0.8) {$\mathcal{Z}^-({\mathcal{I}})$};
            \node  [below,purple]  at (3,0.8) {$\mathcal{Z}^+({\mathcal{I}})$};
            \node  [below,purple]  at (0,0.8) {$\mathcal{Z}({\mathcal{I}})$};
            \node  [below]  at (2.3,0) {$b$};
            \node  [below]  at (-2.3,0) {$a$};
            \end{tikzpicture}
            \end{center}
            \caption{  Different  spectrum sets. }
            \label{fspectrum}
            \end{figure}

    Then, we introduce  the function
    \begin{align}\label{T}
        T(z) &= \prod_{k\in \Delta^-_{z_0}}\left(\dfrac{z-\bar{z}_k}{z-z_k}\right)^2\delta (z)\\
             &=  \prod_{k\in \Delta^-_{z_0}}\left(\dfrac{z-\bar{z}_k}{z-z_k}\right)^2 (z-z_0)^{i\nu(z_0)}e^{i\beta(z,z_0)},
    \end{align}
    where
    \begin{equation}
        \beta(z,z_0)=-\nu(z_0) \log(z-z_0+1)+\int^{z_0}_{-\infty} \dfrac{\nu(s)-\chi(s)\nu(z_0)}{s-z}ds,
    \end{equation}
    here $\chi(s)$ is the characteristic function of the interval $(z_0-1,z_0)$, and log takes the analytic branch along the cut $(-\infty,z_0-1]$.
    \begin{proposition}
        The function $T$ has the following properties
        \begin{itemize}
          \item[(a)]  $T$ is meromorphic in $\mathbb{C}\setminus (-\infty,z_0]$. For each $k\in\Delta^-_{z_0}$, $T(z)$ has double poles at $z_k$ and double zeros at $\bar{z}_k$;
          \item[(b)] For $z\in \mathbb{C}\setminus (-\infty,z_0]$, $\overline{T(\bar{z})}=1/T(z)$;
          \item[(c)] For $z\in (-\infty,z_0]$,
            \begin{equation}
            T_+(z)=T_-(z)(1+|r(z)|^2);
            \end{equation}
          \item[(d)] As $|z|\to \infty$ with $|arg(z)|\leq c<\pi$,
          \begin{equation}\label{expT}
          T(z)=1+\frac{i}{z}\left[ 4\sum_{k\in \Delta^-_{z_0}} {\rm Im} (z_k)-\int^{z_0} _{-\infty}\nu(s)ds\right]+ \mathcal{O}(z^{-2});
          \end{equation}
          \item[(d)]Along the ray $z=z_0+e^{i \phi} \mathbb{R}_+$ where $|\phi|<\pi$, as $z\to z_0$
          \begin{equation}
            |T(z,z_0)-T_0(z_0)(z- z_0)^{i\nu ( z_0)}|\leq c|z-z_0|^{1/2},
          \end{equation}
          where $c$ is a fixed constant.
        \end{itemize}
    \end{proposition}
    \begin{proof}
        The proof of  above properties is similar to the proof of Proposition 3.1 provided by Borghese et al \cite{fNLS}.
   \end{proof}

    Next, we construct a new transformation
    \begin{equation}\label{transm1}
        m^{(1)}(z)=m(z) T(z)^{-\sigma_3}.
    \end{equation}
    From this transformation, we can achieve the following two goals:
    \begin{itemize}
        \item    Renormalize    $m$  such that   $m^{(1)}$   is well behaved as $t\rightarrow \infty$ along arbitrary characteristic;
        \item  Split the residue coefficients  into two sets according to signature of ${\rm Re}(i\theta)$.

        \end{itemize}
   In addition, $m^{(1)}(z)$ satisfies the following RH problem.

   \noindent\textbf{RHP3}.  Find a matrix-valued function $m^{(1)}(z)=m^{(1)}(z;x,t)$ such that
   \begin{itemize}
       \item[(a)]$m^{(1)}(z)$ is analytic in $\mathbb{C}\setminus  \left( \mathbb{R}\cup \mathcal{Z} \cup \bar{\mathcal{Z}} \right)$;
       \item[(b)]$m^{(1)}(z)$ has the following jump condition $m^{(1)}_+(z)=m^{(1)}_-(z)v^{(1)}(z), \hspace{0.5cm}z \in \mathbb{R}$,
       where
       \begin{align}\label{V1}
           v^{(1)}(z)&=\left(\begin{array}{cc}
               1 & \overline{r(z)} T(z)^2 e^{-2it \theta(z)}\\
               0 & 1
           \end{array}\right)\left(\begin{array}{cc}
            1& 0\\
            r(z) T^{-2}(z) e^{2it \theta(z)} & 1
        \end{array}\right),\hspace{0.5cm} z\in (z_0,+\infty);\\
           &=\left(\begin{array}{cc}
            1& 0\\
            \frac{r(z)}{1+|r(z)|^2} T_-^{-2}(z) e^{2it \theta(z)} & 1
        \end{array}\right)\left(\begin{array}{cc}
            1& \frac{\overline{r(z)}}{1+|r(z)|^2} T_+^{-2}(z) e^{-2it \theta(z)}\\
            0 & 1
        \end{array}\right),       \hspace{0.5cm} z\in (-\infty,z_0);
        \end{align}
        \item[(c)]$m^{(1)}(z)=I+\mathcal{O}(z^{-1}),	\quad  \text{as} \;  z \to  \infty $;
        \item[(d)]$m^{(1)}(z)$  satisfies the following residue conditions at double poles $z_k \in \mathcal{Z}$ and $\bar{z}_k \in \bar{\mathcal{Z}}$:
        \begin{align}
            & \res_{z=z_k} m^{(1)}(z)  = \lim_{z\to z_k}    \left(m^{(1)}\right)'(z)\left(\begin{array}{ll}
                    0 & 4 A_k^{-1} \left((T^{-1})''(z_k)\right)^{-2} e^{-2it \theta(z_k)}\\
                     0 & 0
                     \end{array}\right)\\
                 & +     m^{(1)}(z) \left(\begin{array}{ll}
                        0 &  -4 A_k^{-1} \left((T^{-1})''(z_k)\right)^{-2} \left[ B_k+2it \theta'(z_k) + \frac{2 (T^{-1})'''(z_k)}{3 (T^{-1})''(z_k)}  \right]e^{-2it \theta(z_k)} \\
                        0 &   0
                        \end{array} \right) ,  \quad k\in \Delta^-_{z_0};\\
            & \res_{z=z_k} m^{(1)}(z)  = \lim_{z\to z_k}    \left(m^{(1)}\right)'(z)\left(\begin{array}{ll}
                0 & 0\\
                 A_k T^{-2}(z_k) e^{2it \theta(z_k)} & 0
                 \end{array}\right)\\
             & + m^{(1)}(z) \left(\begin{array}{ll}
                0 &  0\\
               A_k T^{-2}(z_k) \left[ B_k+2it \theta'(z_k) -\frac{2 T'(z_k)}{ T(z_k)}  \right]e^{2it \theta(z_k)}  &   0
                \end{array} \right),    \quad k\in \Delta^+_{z_0};\\
           & \res_{z=\bar{z}_k} m^{(1)}(z)  = \lim_{z\to \bar{z}_k}    \left(m^{(1)}\right)'(z)\left(\begin{array}{ll}
            0 & 0\\
             -4 \bar{A}_k^{-1}\left(T''(\bar{z}_k)\right)^{-2} e^{2it \theta(\bar{z_k})}   & 0
             \end{array}\right)\\
          &+ m^{(1)}(z) \left(\begin{array}{ll}
                0 & 0\\
                4 \bar{A}_k^{-1} \left(T''(\bar{z}_k)\right)^{-2} \left[ \bar{B}_k-2it \theta'(\bar{z}_k) + \frac{2 T'''(\bar{z}_k)}{3 T''(\bar{z}_k)}  \right]e^{2it \theta(\bar{z}_k)}  &   0
                \end{array} \right),  \quad k\in \Delta^-_{z_0};\\
            & \res_{z=\bar{z}_k} m^{(1)}(z)  = \lim_{z\to \bar{z}_k}  \left(m^{(1)}\right)'(z)\left(\begin{array}{ll}
                0 & - \bar{A}_k T^2(\bar{z}_k) e^{-2it \theta(\bar{z}_k)}\\
                0 & 0
                 \end{array}\right)\\
              & + m^{(1)}(z) \left(\begin{array}{ll}
                0 &  - \bar{A}_k T^2(\bar{z}_k) \left[ \bar{B}_k-2it \theta'(\bar{z}_k) + \frac{2 T'(\bar{z}_k)}{T(\bar{z}_k)}  \right] e^{-2it \theta(\bar{z}_k)} \\
              0&   0
                \end{array} \right),    \quad k\in \Delta^+_{z_0};
  \end{align}
  Moreover, the coefficients  of the negative second-order term are
 \begin{align}
    & \underset{z=z_k}{P_{-2}} m^{(1)}(z) = \begin{cases}
        \lim_{z\to z_k} m^{(1)}(z) \left(\begin{array}{ll}
            0 & 4 A_k^{-1} \left((T^{-1})''(z_k)\right)^{-2} e^{-2it \theta(z_k)}\\
             0 & 0
             \end{array}\right), & k \in \Delta^-_{z_0};\\
             \lim_{z\to z_k} m^{(1)}(z) \left(\begin{array}{ll}
                0 & 0\\
                 A_k T^{-2}(z_k) e^{2it \theta(z_k)} & 0
                 \end{array}\right), & k \in \Delta^+_{z_0};
        \end{cases}\\
  &  \underset{z=\bar{z}_k}{P_{-2}} m^{(1)}(z)  = \begin{cases}
        \lim_{z\to \bar{z}_k} m^{(1)}(z) \left(\begin{array}{ll}
            0 & 0\\
            -4 \bar{A}_k^{-1}\left(T''(\bar{z}_k)\right)^{-2} e^{2it \theta(\bar{z}_k)} & 0
             \end{array}\right), & k \in \Delta^-_{z_0};\\
             \lim_{z\to \bar{z}_k} m^{(1)}(z) \left(\begin{array}{ll}
                0 & - \bar{A}_k T^2(\bar{z}_k) e^{-2it \theta(\bar{z}_k)}\\
                0 & 0
                 \end{array}\right), & k \in \Delta^+_{z_0}.
        \end{cases}
\end{align}
  \end{itemize}
  \begin{proof}
    The analyticity, jump condition and asymptotic behavior of $m^{(1)}(z)$ are easily to be proven.
    The difficulty lies in the calculation of residue conditions.
    We first consider  poles $z_k \in \mathcal{Z}$ in the upper half complex plane and denote $m(z)=\left(m_1(z), m_2(z) \right) $, then
    $$m^{(1)}(z)=\left(m^{(1)}_1(z), m^{(1)}_2(z) \right) =   \left(m_1(z)T^{-1}(z), m_2(z)T(z) \right)=\left(  \frac{\mu^{-}_{1}(z)}{s_{11}(z)}T^{-1}(z),   \mu^{+}_{2}(z)    T(z)   \right).$$
    \begin{itemize}
        \item[(i)] For $k \in \Delta^-_{z_0}$ and $z_k \in \mathcal{Z}$, $z_k$ is the double poles of $m^{(1)}_1$ and $T$, but $m^{(1)}_2$ and $T^{-1}$ are analytic at the point $z_k$ with $T^{-1}(z_k)=0$, then
        \begin{align}
            & m^{(1)}(z_k)  = \frac{1}{2} A_k e^{2it \theta(z_k)}m_2(z_k) (T^{-1})''(z_k),\label{mp1m2}\\
          &\res_{z=z_k}m^{(1)}_2(z)= m'_2(z_k)  2 \left(T^{-1}\right)''(z_k) - m_2(z_k) \frac{2\left(T^{-1}\right)'''(z_k) }{3 \left(\left(T^{-1}\right)''(z_k)\right)^2}. \label{mp12}
         \end{align}
    where $\hat{T}(z)=T(z) (z-z_k)^2$.
    Next, we calculate the derivative of $m_1^{(1)}$
       \begin{equation}\label{dm11}
        \left(m_1^{(1)}\right)'(z_k)=\left(\mu^-_{1}\right)'(z_k) \frac{T^{-1}(z_k)}{s_{11}(z_k)}+ \mu^-_{1}(z_k) \left( \frac{T^{-1}}{s_{11}}\right)'(z_k).
        \end{equation}
    From the Taylor expansion, we find
    \begin{equation}
        \frac{T^{-1}(z)}{s_{11}(z)}=\frac{\left(T^{-1}\right)''(z_k)}{s''_{11}(z_k)}+ \left(  \frac{\left(T^{-1}\right)'''(z_k)}{3s''_{11}(z_k)} -\frac{\left(T^{-1}\right)''(z_k) s'''_{11}(z_k) }{3(s''_{11}(z_k))^2}             \right)(z-z_k)+\cdots
    \end{equation}
    Thus, we know
    \begin{equation}\label{T-1s11}
        \left( \frac{T^{-1}}{s_{11}}\right)'(z_k) =  \frac{\left(T^{-1}\right)'''(z_k)}{3s''_{11}(z_k)} -\frac{(T^{-1})''(z_k) s'''_{11}(z_k) }{3\left(s''_{11}(z_k)\right)^2}.
    \end{equation}
  Combing (\ref{u01-}), (\ref{u101-}) and (\ref{dm11}) with (\ref{T-1s11}), we obtain
  \begin{align}\label{dmp1m2mp2}
     & \left(m_1^{(1)}\right)'(z_k)= b_k \frac{\left(T^{-1}\right)''(z_k)}{s''_{11}(z_k)} e^{2it \theta(z_k)} \left(\mu^+_{2}\right)'(z_k) \notag\\
     & +\left[  \frac{(T^{-1})''(z_k)}{s''_{11}(z_k)} \left( 2it \theta'(z_k)b_k +d_k     \right)     +b_k \left(  \frac{(T^{-1})'''(z_k)}{3s''_{11}(z_k)} -\frac{(T^{-1})''(z_k) s'''_{11}(z_k) }{3(s''_{11}(z_k))^2}       \right)  \right] e^{2it \theta(z_k)} \mu^+_{2}(z_k).
  \end{align}
    Substituting (\ref{mp1m2}) and (\ref{dmp1m2mp2}) into (\ref{mp12}) , we find
    \begin{align}
       & \res_{z=z_k} m^{(1)}_2(z) = 4 A_k^{-1} \left((T^{-1})''(z_k)\right)^{-2} e^{-2it \theta(z_k)} \left(m^{(1)}_2\right)'(z_k)  \notag\\
       & -4 A_k^{-1} \left((T^{-1})''(z_k)\right)^{-2} \left[ B_k+2it \theta'(z_k) + \frac{2 (T^{-1})'''(z_k)}{3 (T^{-1})''(z_k)}  \right]e^{-2it \theta(z_k)} m^{(1)}_2(z_k),
    \end{align}
    where we have used the fact $A_k=\frac{2b_k}{s''_{11}(z_k)}$. Finally,  we obtain  the corresponding residue condition for $m^{(1)}(z)$.

    Then, we calculate the coefficient of $(z-z_k)^{-2}$ in the Laurent  expansion of $m^{(1)}$. We still consider this condition according to the order of the columns of $m^{(1)}$.
   \begin{align}
   & \underset{z=z_k}{P_{-2}} m^{(1)}_1(z)=0,\\
   & \underset{z=z_k}{P_{-2}} m^{(1)}_2(z)=m_2(z_k) \underset{z=z_k}{P_{-2}} T(z)=2 m_2(z_k)  \left( (T^{-1})''(z_k) \right)^{-1}.\label{p-2mp2}
   \end{align}
   Plugging (\ref{mp1m2}) into (\ref{p-2mp2}), it is straightforward to find
   \begin{align}
    \underset{z=z_k}{P_{-2}} m^{(1)}(z)=\left( 0, 4 A_k^{-1} \left((T^{-1})''(z_k)\right)^{-2} e^{-2it \theta(z_k)}      \right).
   \end{align}

    \item[(ii)]For $k \in \Delta^+_{z_0}$ and $z_k \in \mathcal{Z}$, $T$ and $T^{-1}$ is analytic at the point $z_k$. In this case, the residue condition of the first column of $m^{(1)}$ is
    \begin{align}
        \res_{z=z_k} m^{(1)}_1(z) &=  \res_{z=z_k}  m_1(z)\cdot T^{-1}(z_k)+\res_{z=z_k}  m_1(z) \left( T^{-1}(z)-T^{-1}(z_k)\right)  \notag\\
        &=\res_{z=z_k}  m_1(z)\cdot T^{-1}(z_k) + \lim_{z\to z_k} m_1 \left(T^{-1}\right)'(z_k)(z-z_k)^2.
        \end{align}
    Since
    \begin{equation}
        \lim_{z\to z_k} m_1 \left(T^{-1}\right)'(z_k)(z-z_k)^2= 2 \mu^-_{1}(z_k)\left(T^{-1}\right)'(z_k) /s''_{11}(z_k),
    \end{equation}
    \begin{equation}
        \left(m^{(1)}_2\right)'(z)= m'_2(z) T(z) + m_2(z) T'(z),
    \end{equation}
    we give the expression
    \begin{align}
        \res_{z=z_k} m^{(1)}_1(z)&= A_k T^{-2}(z_k) e^{2it\theta(z_k)}  (m^{(1)}_2)'(z_k)  \notag\\
        &+  A_k T^{-2}(z_k) \left[ B_k+2it \theta'(z_k) -\frac{2 T'(z_k)}{ T(z_k)}  \right]e^{2it \theta(z_k)} m^{(1)}_2(z_k),
    \end{align}
    in which we have used equations (\ref{u01-}) and (\ref{AkBk1}).
    Because of the analyticity of $m_2(z)$ and $T(z)$ at the point $z_k$, $ \res_{z=z_k} m^{(1)}_2(z)=0$.
    Thus, we find the expression of the residue condition in this case.

    In addition,  we can obtain the coefficient of $(z-z_k)^{-2}$ in the Laurent  expansion of $m^{(1)}$ directly because $T^{-1}$ is analytic at the point $z_k$
    \begin{equation}
       \underset{z=z_k}{P_{-2}} m^{(1)}(z)=\underset{z=z_k}{P_{-2}} m(z) \cdot T^{-1}(z_k)=\lim_{z\to z_k} m^{(1)} \left(\begin{array}{ll}
            0 & 0\\
             A_k T^{-2}(z_k) e^{2it \theta(z_k)} & 0
             \end{array}\right).
        \end{equation}

    Using the same method, we can obtain the corresponding conditions for $\bar{z}_k\in \bar{\mathcal{Z}}$ in the lower half plane.
    \end{itemize}
  \end{proof}

 \subsection{Continuous extensions of jump matrix}
    In this section, we make continuous extension to the scattering data of the jump matrix $v^{(1)}$ and construct  a new  transformation from $m^{(1)}$ to $m^{(2)}$.
    The transformed $m^{(2)}$ satisfies the following properties:
    \begin{itemize}
        \item     $m^{(2)}(z)$   has no jump on $\mathbb{R}$ and  matches   $m^{(pc)}(z)$ model,  which is given and analyzed in Section 5.1.2,  on a new  contour $\Sigma^{(2)}$ which is defined in (\ref{sigma2}).
        \item   The norm of the function, which is introduced by this transformation, has been controlled so that the $\overline{\partial}$-contribution to the long-time asymptotics of $q(x,t)$ can be ignored.
        \item  The residues are unaffected by the transformation.
        \end{itemize}

    To make  continuous extension, we first define a new contour $\Sigma^{(2)}$
    \begin{equation}\label{sigma2}
        \Sigma^{(2)}= \Sigma_1 \cup \Sigma_2 \cup \Sigma_3 \cup \Sigma_4,
    \end{equation}
   where $\Sigma_k$ are  given in (\ref{Sigma0}).
   Then, the real axis  $\mathbb{R}$  and the contour $\Sigma^{(2)}$  separate complex plane  $\mathbb{C}$  into six open sectors denoted  by $\Omega_k$, $k=1,...,6$, depicted in Figure \ref{fig5}.

   Second, let
   \begin{equation}
     \rho =\frac{1}{2}\min_{\lambda, \mu \in \mathcal{Z}\cup \bar{\mathcal{Z}};\lambda\neq\mu}|\lambda -\mu|.
   \end{equation}
   For any point $z_k=x_k+iy_k \in \mathcal{Z}$, we have $\bar{z_k}=x_k-iy_k \in \bar{\mathcal{Z}}$. Thus, ${\rm dist}(\mathcal{Z},\mathbb{R}) \ge \rho$.
   Suppose that $\chi_\mathcal{Z} \in C_0^\infty(\mathbb{C},[0,1])$ is the characteristic function defined in the neighborhood of discrete spectrum
   \begin{equation}
    \chi_\mathcal{Z}(z)=\Bigg\{\begin{array}{ll}
    1 &\text{dist}\left(z,\mathcal{Z}\cup \bar{\mathcal{Z}}\right)<\rho/3,\\
    0 &\text{dist}\left(z,\mathcal{Z}\cup \bar{\mathcal{Z}}\right)>2\rho/3.\\
    \end{array}
    \end{equation}

    Finally, we introduce a transformation $\mathcal{R}^{(2)}$ to obtain a mixed  $\bar{\partial}$-RH problem.
    \begin{equation}\label{transm2}
        m^{(2)}(z)=m^{(1)}(z)\mathcal{R}^{(2)}(z),
    \end{equation}
    where  $\mathcal{R}^{(2)}(z)$ is defined as follows:
    \begin{equation}
        \mathcal{R}^{(2)}(z)=\left\{\begin{array}{lllll}
        \left(\begin{array}{cc}
        1 & 0\\
        R_1(z)e^{2it\theta} & 1
        \end{array}\right)^{-1}=W_R^{-1}, & z\in \Omega_1,\\
        \\
        \left(\begin{array}{cc}
            1 & R_3(z)e^{-2it\theta}\\
            0 & 1
            \end{array}\right)^{-1}=U^{-1}_R, & z\in \Omega_3,\\
            \\
        \left(\begin{array}{cc}
        1 & 0\\
        R_4(z)e^{2it\theta} & 1
       \end{array}\right)=U_L,  &z\in \Omega_4,\\
        \\
        \left(\begin{array}{cc}
            1 & R_6(z)e^{-2it\theta}\\
            0 & 1
            \end{array}\right)=W_L, & z\in \Omega_6,\\
            \\

        I,  &z\in \Omega_2\cup\Omega_5;\\
        \end{array}\right.
        \end{equation}
        where the function $R_j$, $j=1,3,4,6$, is defined in following proposition, depicted in Figure \ref{fig5}.

        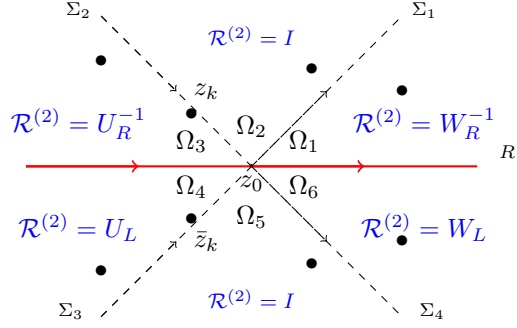
\begin{figure}
        \begin{center}
        \begin{tikzpicture}
        \draw [red,thick ](-3,0)--(3,0);
        \draw [red,thick,-> ](-3,0)--(-1.5,0);
        \draw[red,thick, -> ](0,0)--(1.5,0);
        \draw [dashed] (-2,-2)--(2,2);
        \draw [dashed][ -> ](-2,-2)--(-1,-1);
        \draw [dashed] [  -> ](0,0)--(1,1);
        \draw [dashed] (-2,2)--(2,-2);
        \draw [dashed] [ -> ](-2,2)--(-1,1);
        \draw [dashed] [  -> ] (0,0)--(1,-1);
        \node  [below]  at (0,0) {$z_0$};
        \node  [below]  at (2.3,2.3) {\scriptsize $\Sigma_1$};
        \node  [below]  at (-2.3,2.3) {\scriptsize $\Sigma_2$};
        \node  [below]  at (-2.4,-1.7) {\scriptsize $\Sigma_3$};
        \node  [below]  at (2.4,-1.7) {\scriptsize $\Sigma_4$};
        \node  [below]  at (3.4,0.4) {\scriptsize $R$};
        \node    at (0.8,1.3)  {$\bullet$};
        \node    at (0.8,-1.3)  {$\bullet$};
        \node    at (2,1 )  {$\bullet$};
        \node    at (2,-1)  {$\bullet$};
        \node    at (-0.8,0.7 )  {$\bullet$};
        \node    at (-0.8,-0.7)  {$\bullet$};
        \node    at (-2,1.4 )  {$\bullet$};
        \node    at (-2,-1.4)  {$\bullet$};
        \node    at (-0.6,1)  {$z_k$};
        \node    at (-0.6,-1)  {$\bar z_k$};
        \node  [below]  at (0.7,0.6) {$\Omega_1$};
        \node  [below]  at (0, 0.8) {$\Omega_2$};
        \node  [below]  at (-0.8, 0.6) {$\Omega_3$};
        \node  [below]  at (-0.8,-0) {$\Omega_4$};
        \node  [below]  at (0, -0.4) {$\Omega_5$};
        \node  [below]  at (0.7,0) {$\Omega_6$};
        \node    at (0,0)  {$\cdot$};
        \node  [below,blue]  at (-2.3, 0.9) {  $\mathcal{R}^{(2)}= U_R^{-1}$};
        \node  [below,blue]  at (-2.3,-0.5) {$\mathcal{R}^{(2)}= U_L$};
        \node  [below,blue]  at (2.3,0.9) {$\mathcal{R}^{(2)}= W_R^{-1}$};
        \node  [below,blue]  at (2.3,-0.5) {$\mathcal{R}^{(2)}= W_L$};
        \node [thick,blue] [below]  at (0,2 ) {\footnotesize $ \mathcal{R}^{(2)}=I $};
        \node [thick,blue] [below]  at (0,-1.5) {\footnotesize $\mathcal{R}^{(2)}=I$};
        \end{tikzpicture}
        \end{center}
        \caption{ Definition of $\mathcal{R}^{(2)}$ in different domains.}
        \label{fig5}
        \end{figure}

    \begin{proposition}\label{proR}
       There exists a function $R_j$: $\bar{\Omega}_j\to C$, $j=1,3,4,6$ such that
       \begin{align}
       &R_1(z)=\Bigg\{\begin{array}{ll}
       r(z)T(z)^{-2}, & z\in (z_0,\infty),\\
       r(z_0)T_0(z_0)^{-2}(z-z_0)^{-2i\nu(z_0)}(1-\chi_\mathcal{Z}(z)),  &z\in \Sigma_1,\\
       \end{array} \\
       &R_3(z)=\Bigg\{\begin{array}{ll}
       \dfrac{\overline{r(z)}T_+(z)^2}{1+|r(z)|^2}, & z \in (-\infty,z_0),\\
       \dfrac{\overline{r(z_0)}T_0(z_0)^2}{1+|r(z_0)|^2}(z-z_0)^{2i\nu(z_0)}(1-\chi_\mathcal{Z}(z)),  &z\in \Sigma_2,\\
       \end{array} \\
       &R_4(z)=\Bigg\{\begin{array}{ll}
       \dfrac{r(z)T_-(z)^{-2}}{1+|r(z)|^2}, & z \in (-\infty,z_0), \\
       \dfrac{r(z_0)T_0(z_0)^{-2}}{1+|r(z_0)|^2}(z-z_0)^{-2i\nu(z_0)}(1-\chi_\mathcal{Z}(z)),  &z\in \Sigma_3,\\
       \end{array} \\
       &R_6(z)=\Bigg\{\begin{array}{ll}
        \overline{r(z)}T(z)^{2}, & z\in (z_0,\infty),\\
        \overline{r(z_0)}T_0(z_0)^{2}(z-z_0)^{2i \nu(z_0)}(1-\chi_\mathcal{Z}(z)),  &z\in \Sigma_4,\\
       \end{array}
       \end{align}	
      and $R_j$  admit estimates
      \begin{align}
       &|R_j(z)|\lesssim\sin^2(\arg(z-z_0))+\langle \text{\rm{Re}}(z)\rangle^{-1/2},\label{R}\\
       &|\bar{\partial}R_j(z)|\lesssim|\bar{\partial}\chi_\mathcal{Z}(z)|+|r'(\text{\rm{Re}}z)|+|z-z_0|^{-1/2},\label{dbarRj}\\
       & \bar{\partial}R_j(z)=0,\hspace{0.5cm}\text{if } z\in \Omega_2\cup\Omega_5\; \text{or} \;\text{\rm{dist}}(z,\mathcal{Z}\cup \bar{\mathcal{Z}})<\rho/3.
    \end{align}
   \end{proposition}
   The proof of above proposition is the same as that  in \cite{fNLS} because  the form of residue condition doesn't affect this transform $\mathcal{R}^{(2)}(z)$ .

   Therefore, $m^{(2)}(z)$ satisfies the  mixed $\bar{\partial}$-RH problem as follows:

   \begin{figure}
    \begin{center}
    \begin{tikzpicture}
    \draw[yellow!20, fill=yellow!20] (0,0)--(-2,-2)--(-2.6,-2)--(-2.6, 2)--(-2, 2 )--(0,0);
    \draw[yellow!20, fill=yellow!20] (0,0)--( 2,-2)--( 2.6,-2)--( 2.6, 2)--( 2, 2 )--(0,0);
    \draw[dashed] (-3,0)--(3,0);
    \draw[dashed]  [ -> ](-3,0)--(-1.5,0);
    \draw[dashed]  [   -> ](0,0)--(1.5,0);
    \draw[red,thick ](-2,-2)--(2,2);
    \draw[red,thick,-> ](-2,-2)--(-1,-1);
    \draw[red,thick, -> ](0,0)--(1,1);
    \draw[red,thick ](-2,2)--(2,-2);
    \draw[red,thick,-> ](-2,2)--(-1,1);
    \draw[red,thick, -> ](0,0)--(1,-1);
    \node  [below]  at (2.3,2 ) {$\Sigma_1$};
    \node  [below]  at (-2.3,2 ) {$\Sigma_2$};
    \node  [below]  at (-2.2 ,-1.6) {$\Sigma_3$};
    \node  [below]  at (1.7,-1.7) {$\Sigma_4$};
    \node  [below]  at (0,-0.1) {$z_0$};
    \path [fill=white] (2,1) circle [radius=0.2];
    \path [fill=white] (2,-1) circle [radius=0.2];
    \path [fill=white] (-0.8,0.7) circle [radius=0.2];
    \path [fill=white] (-0.8,-0.7) circle [radius=0.2];
    \path [fill=white] (-2,1.4 )  circle [radius=0.2];
    \path [fill=white] (-2,-1.4 ) circle [radius=0.2];
    \node    at (0.8,1.3)  {$\bullet$};
    \node    at (0.8,-1.3)  {$\bullet$};
    \node    at (2,1 )  {$\bullet$};
    \node    at (2,-1)  {$\bullet$};
    \node    at (-0.8,0.7 )  {$\bullet$};
    \node    at (-0.8,-0.7)  {$\bullet$};
    \node    at (-2,1.4 )  {$\bullet$};
    \node    at (-2,-1.4)  {$\bullet$};
    \node    at (-0.6,1)  {$z_k$};
    \node    at (-0.6,-1)  {$\bar z_k$};
    \node  [below]  at (0.7,0.6) {$\Omega_1$};
    \node  [below]  at (0, 0.8) {$\Omega_2$};
    \node  [below]  at (-0.8, 0.6) {$\Omega_3$};
    \node  [below]  at (-0.8,-0) {$\Omega_4$};
    \node  [below]  at (0, -0.4) {$\Omega_5$};
    \node  [below]  at (0.7,0) {$\Omega_6$};

    \node [thick] [below]  at (3.2, 1.2) {\tiny $ \left(\begin{array}{cc} 1&0\\ R_1  e^{2it\theta} &1\end{array}  \right)  $};
    \node [thick] [below]  at (-3,-0.3) {\tiny $ \left(\begin{array}{cc} 1&0\\ R_4e^{2it\theta} &1\end{array}  \right) $};
    \node [thick] [below]  at (-3,1.2) {\tiny $\left(\begin{array}{cc} 1&  R_3e^{-2it\theta} \\ 0&1\end{array}  \right)$};
    \node [thick] [below]  at (3.3,-0.3) {\tiny $\left(\begin{array}{cc} 1&  R_6e^{-2it\theta}  \\ 0&1\end{array}  \right)$};
    \end{tikzpicture}
    \end{center}
    \caption{    Jump matrix  $v^{(2)}$. Yellow  parts support  $\bar\partial$ derivative:   $\bar\partial \mathcal{R}^{(2)}\not=0$;
    White parts don't support $\bar\partial$ derivative: $\bar\partial \mathcal{R}^{(2)}=0$.}
    \label{fv2jump}
    \end{figure}
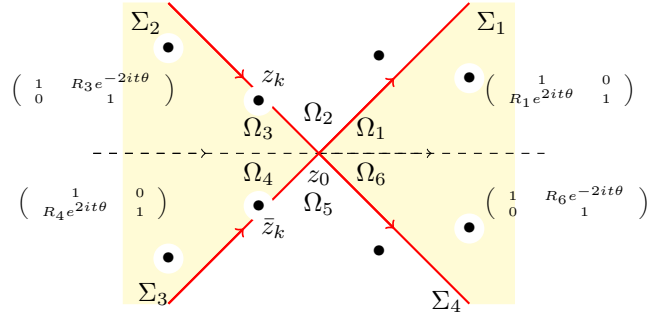

   \noindent\textbf{RHP4}.  Find a matrix-valued function $m^{(2)}(z)=m^{(2)}(z;x,t)$ which satisfies
   \begin{itemize}
       \item[(a)]$m^{(2)}(z)$ is continuous in $\mathbb{C}\setminus  \left( \Sigma^{(2)}\cup \mathcal{Z} \cup \bar{\mathcal{Z}} \right)$.
       \item[(b)]$m^{(2)}(z)$ has the following jump condition $m^{(2)}_+(z)=m^{(2)}_-(z)v^{(2)}(z), \hspace{0.5cm}z \in \Sigma^{(2)}$,
       where
       \begin{align}
        v^{(2)}(z) & = \left(\mathcal{R}^{(2)}_-\right)^{-1}(z) v^{(1)}(z)\mathcal{R}^{(2)}_+(z)  \notag\\
            &= \begin{cases}
            \left(\begin{array}{cc}
            1 & 0\\
            R_1(z)e^{2it\theta} & 1
            \end{array}\right), & z\in \Sigma_1,\\
            \left(\begin{array}{cc}
                1 & R_3(z)e^{-2it\theta}\\
                0 & 1
                \end{array}\right), & z\in \Sigma_2,\\
            \left(\begin{array}{cc}
            1 & 0\\
            R_4(z)e^{2it\theta} & 1
           \end{array}\right),  &z\in \Sigma_3,\\
            \left(\begin{array}{cc}
                1 & R_6(z)e^{-2it\theta}\\
                0 & 1
                \end{array}\right), & z\in \Sigma_4.
            \end{cases} \label{v2}
        \end{align}
        See Figure \ref{fv2jump}.

        \item[(c)]$m^{(2)}(z)=I+\mathcal{O}(z^{-1}),	\qquad  \text{as} \;  z \to  \infty $.
        \item[(d)]For $\mathbb{C}\setminus (\Sigma^{(2)}\cup \mathcal{Z}\cup \bar{\mathcal{Z}})$,
        \begin{equation}
            \bar{\partial}m^{(2)}(z)=m^{(2)} (z)\bar{\partial} \mathcal{R}^{(2)}(z),
        \end{equation}
        where
        \begin{equation}
            \bar{\partial} \mathcal{R}^{(2)}(z)=\left\{\begin{array}{lllll}
            \left(\begin{array}{cc}
            1 & 0\\
            -\bar{\partial}R_1(z)e^{2it\theta} & 1
            \end{array}\right), & z\in \Omega_1,\\
            \\
            \left(\begin{array}{cc}
                1 & -\bar{\partial}R_3(z)e^{-2it\theta}\\
                0 & 1
                \end{array}\right), & z\in \Omega_3,\\
                \\
            \left(\begin{array}{cc}
            1 & 0\\
            \bar{\partial}R_4(z)e^{2it\theta} & 1
           \end{array}\right),  &z\in \Omega_4,\\
            \\
            \left(\begin{array}{cc}
                1 &\bar{\partial} R_6(z)e^{-2it\theta}\\
                0 & 1
                \end{array}\right), & z\in \Omega_6,\\
                \\
                \left(\begin{array}{cc}
                    0& 0\\
                    0 & 0
                    \end{array}\right), &z\in \Omega_2\cup\Omega_5.\\
            \end{array}\right.
            \end{equation}
        \item[(e)]$m^{(2)}(z)$  satisfies the following condition at double poles $z_k \in \mathcal{Z}$ and $\bar{z}_k \in \bar{\mathcal{Z}}$:
        \begin{align}
            & \res_{z=z_k} m^{(2)}(z)  = \lim_{z\to z_k}   \left(m^{(2)}\right)'(z)\left(\begin{array}{ll}
                0 & 4 A_k^{-1} \left((T^{-1})''(z_k)\right)^{-2} e^{-2it \theta(z_k)}\\
                 0 & 0
                 \end{array}\right)\\
            & + \lim_{z\to z_k} m^{(2)}(z) \left(\begin{array}{ll}
                0 &  -4 A_k^{-1} \left((T^{-1})''(z_k)\right)^{-2} \left[ B_k+2it \theta'(z_k) + \frac{2 (T^{-1})'''(z_k)}{3 (T^{-1})''(z_k)}  \right]e^{-2it \theta(z_k)} \\
                0 &   0
                \end{array} \right) ,  \quad k\in \Delta^-_{z_0};\\
            & \res_{z=z_k} m^{(2)}(z)  = \lim_{z\to z_k} \left(m^{(2)}\right)'(z)\left(\begin{array}{ll}
                0 & 0\\
                 A_k T^{-2}(z_k) e^{2it \theta(z_k)} & 0
                 \end{array}\right) \\
            & + \lim_{z\to z_k} m^{(2)}(z) \left(\begin{array}{ll}
                0 &  0\\
               A_k T^{-2}(z_k) \left[ B_k+2it \theta'(z_k) -\frac{2 T'(z_k)}{ T(z_k)}  \right]e^{2it \theta(z_k)}  &   0
                \end{array} \right),    \quad k\in \Delta^+_{z_0};  \\
           & \res_{z=\bar{z}_k} m^{(2)}(z)  = \lim_{z\to \bar{z}_k}  \left(m^{(2)}\right)'(z)\left(\begin{array}{ll}
            0 & 0\\
             -4 \bar{A_k}^{-1}\left(T''(\bar{z}_k)\right)^{-2} e^{2it \theta(\bar{z}_k)}   & 0
             \end{array}\right)\\
           &+ \lim_{z\to \bar{z}_k} m^{(2)}(z) \left(\begin{array}{ll}
                0 & 0\\
                4 \bar{A}_k^{-1} \left(T''(\bar{z}_k)\right)^{-2} \left[ \bar{B}_k-2it \theta'(\bar{z}_k) + \frac{2 T'''(\bar{z}_k)}{3 T''(\bar{z}_k)}  \right]e^{2it \theta(\bar{z}_k)}  &   0
                \end{array} \right),  \quad k\in \Delta^-_{z_0};\\
            & \res_{z=\bar{z}_k} m^{(2)}(z)  = \lim_{z\to \bar{z}_k}   \left(m^{(2)}\right)'(z)\left(\begin{array}{ll}
                0 & - \bar{A}_k T^2(\bar{z}_k) e^{-2it \theta(\bar{z}_k)}\\
                0 & 0
                 \end{array}\right)\\
            & + \lim_{z\to \bar{z}_k} m^{(2)}(z) \left(\begin{array}{ll}
                0 &  - \bar{A}_k T^2(\bar{z}_k) \left[ \bar{B}_k-2it \theta'(\bar{z}_k) + \frac{2 T'(\bar{z}_k)}{T(\bar{z}_k)}  \right] e^{-2it \theta(\bar{z}_k)} \\
              0&   0
                \end{array} \right),    \quad k\in \Delta^+_{z_0};\\
          \end{align}
 \begin{align}
   & \underset{z=z_k}{P_{-2}} m^{(2)}(z) = \begin{cases}
        \lim_{z\to z_k} m^{(2)}(z) \left(\begin{array}{ll}
            0 & 4 A_k^{-1} \left(\left(T^{-1}\right)''(z_k)\right)^{-2} e^{-2it \theta(z_k)}\\
             0 & 0
             \end{array}\right), & k \in \Delta^-_{z_0};\\
             \lim_{z\to z_k} m^{(2)}(z) \left(\begin{array}{ll}
                0 & 0\\
                 A_k T^{-2}(z_k) e^{2it \theta(z_k)} & 0
                 \end{array}\right), & k \in \Delta^+_{z_0};
        \end{cases}\\
   & \underset{z=\bar{z}_k}{P_{-2}} m^{(2)}(z) = \begin{cases}
        \lim_{z\to \bar{z}_k} m^{(2)}(z) \left(\begin{array}{ll}
            0 & 0\\
            -4 \bar{A}_k^{-1}\left(T''(\bar{z}_k)\right)^{-2} e^{2it \theta(\bar{z}_k)} & 0
             \end{array}\right), & k \in \Delta^-_{z_0};\\
             \lim_{z\to \bar{z}_k} m^{(2)}(z) \left(\begin{array}{ll}
                0 & - \bar{A}_k T^2(\bar{z}_k) e^{-2it \theta(\bar{z}_k)}\\
                0 & 0
                 \end{array}\right), & k \in \Delta^+_{z_0}.
        \end{cases}
    \end{align}
  \end{itemize}
  \begin{proof}
    (a)-(d) are easy to be checked, so here we only give a brief proof to (e).
    It is sufficient to prove the case where $k \in \Delta^+_{z_0}$ and $z_k\in \mathbb{C}^+$, since the proof of others is similar.
    In this case, any $z_k \in \Omega_1$ is not the pole of $\mathcal{R}^{(2)}$, so
   \begin{align*}
    & \res_{z=z_k} m^{(2)}(z)  = \lim_{z\to z_k}  \left(m^{(1)}\right)' \mathcal{R}^{(2)} \left(\mathcal{R}^{(2)}\right)^{-1}  \left(\begin{array}{ll}
        0 & 0\\
         A_k T^{-2}(z_k) e^{2it \theta(z_k)} & 0
         \end{array}\right)  \mathcal{R}^{(2)}\\
      & + \lim_{z\to z_k}    m^{(1)}\mathcal{R}^{(2)} \left(\mathcal{R}^{(2)}\right)^{-1} \left(\begin{array}{ll}
        0 &  0\\
       A_k T^{-2}(z_k) \left[ B_k+2it \theta'(z_k) -\frac{2 T'(z_k)}{ T(z_k)}  \right]e^{2it \theta(z_k)}  &   0
        \end{array} \right)  \mathcal{R}^{(2)}.
    \end{align*}
   Substitute $\left(m^{(2)}\right)'=\left(m^{(1)}\right)'\mathcal{R}^{(2)} + m^{(1)} \left(\mathcal{R}^{(2)}\right)' $ and $\left(\mathcal{R}^{(2)}\right)'= 0$ into the above equation, and  we  finish the proof.
   \end{proof}

\section{Decomposition of the mixed $\bar{\partial}$-RH problem }
    In this section, we will find the solution of the mixed $\bar{\partial}$-RH problem $m^{(2)}$ as follows:

    Step 1: Separate zero and non-zero parts of $\bar{\partial} \mathcal{R}^{(2)}$.
    Thus, we decompose $m^{(2)}$ into a pure Riemann-Hilbert problem with $\bar{\partial} \mathcal{R}^{(2)}=0$, which we denote by $m^{(2)}_{RHP}$, and a pure  $\bar{\partial}$ problem with $\bar{\partial} \mathcal{R}^{(2)} \neq 0$, which we denote by $m^{(3)}$.
    \begin{align}
        m^{(2)}=
       \left\{\begin{matrix}
       \bar\partial \mathcal{R}^{(2)} =0\rightarrow  \ m^{(2)}_{RHP},\qquad\qquad\qquad\qquad \vspace{4mm}\cr
       \bar\partial \mathcal{R}^{(2)} \not=0\rightarrow  \ m^{(3)}=m^{(2)} \left( m^{(2)}_{RHP}\right)^{-1}.
       \end{matrix}\right.
       \end{align}
       The RH problem for the $m^{(2)}_{RHP}$ is as follows:\\
       \noindent\textbf{RHP5}.  Find a matrix-valued function $m^{(2)}_{RHP}(z)$ which satisfies
       \begin{itemize}
           \item[(a)]$m^{(2)}_{RHP}(z)$ is analytic in $\mathbb{C}\setminus  \left( \Sigma^{(2)}\cup \mathcal{Z} \cup \bar{\mathcal{Z}} \right)$;
           \item[(b)]$m^{(2)}_{RHP}(z)$ has the following jump condition $m^{(2)}_{RHP+}(z)=m^{(2)}_{RHP-}(z)v^{(2)}(z), \hspace{0.5cm}z \in \Sigma^{(2)}$,
           where $v^{(2)}(z)$ has been given by (\ref{v2});
            \item[(c)]$m^{(2)}_{RHP}(z) \to I,	\quad  \text{as} \;  z \to  \infty $;
            \item[(d)]$\bar{\partial}\mathcal{R}^{(2)}(z)=0$, \quad for $ z\in \mathbb{C}$;
            \item[(e)] The residue condition and the  coefficient of the negative twice power of the Laurent expansion have the same form as $m^{(2)}$ with $m^{(2)}_{RHP}$ replacing $m^{(2)}$.
  \end{itemize}

    Step 2: To  prove the existence  of the solution $m^{(2)}_{RHP}$, we separate the jump line from the pole.
    Suppose $\epsilon_{z_0}=\left\{ z:|z-z_0|<\rho/2 \right\}$.
    Let $m^{(2)}_{RHP}$ be further decomposed into two parts:
    \begin{equation}\label{m2rhp}
        m^{(2)}_{RHP}(z)=\left\{\begin{array}{ll}
        E(z)m^{(out)}(z), & z\notin \epsilon_{z_0},\\
        E(z)m^{(z_0)}(z),  &z\in \epsilon_{z_0}.\\
        \end{array}\right.
        \end{equation}
    The outer model $m^{(out)}(z)$ is constructed by ignoring the jumps in  RHP5, which can be approximated by solitons on the discrete spectrum.  The inner model $m^{(z_0)}(z)$ has the same jump with RHP5, which can be approximated by the parabolic cylinder model in continuous spectrum.

    Step 3: Find the  solution and its asymptotic behavior of the pure  $\bar{\partial}$ problem $m^{(3)}(z)$.

  \subsection{The pure RH problem and constructions of its solution}

\subsubsection{The construction of outer  model }
    By definition (\ref{m2rhp}), $m^{(out)}(z)$ is the solution of $m^{(2)}(z)$ in the soliton region, which satisfies the following RH problem.

\noindent\textbf{RHP6}.  Find a matrix-valued function $m^{(out)}(z)$ which satisfies
  \begin{itemize}
      \item[(a)]$m^{(out)}(z)$ is analytic in $\mathbb{C}\setminus  \left( \mathcal{Z} \cup \bar{\mathcal{Z}} \right)$;
      \item[(b)]$m^{(out)}(z)=I + O(z^{-1}),	\quad   z \to  \infty $;
      \item[(c)]$m^{(out)}(z)$ has double poles at each $z_k\in \mathcal{Z}$ and $\bar{z}_k \in \bar{\mathcal{Z}}$, which satisfies the residue relations in (e) of RHP4 with $m^{(out)}(z)$ replacing $m^{(2)}(z)$.
  \end{itemize}

   In order to show the existence and uniqueness of solution of $m^{(out)}(z)$,  we first consider the reflectionless case of the RHP1.
   In this case, $r(z)=0$ and $v(z)=I$, then $m_+=m_-$. Thus, RHP1 of NLS equation has no jumps in the whole plane and is analytic in $\mathbb{C}$ except for $z_k \in \mathcal{Z}$ and $\bar{z}_k \in \bar{\mathcal{Z}}$.
   The RHP1 can be equivalently rewritten as the following solvable RH problem:

   \noindent\textbf{RHP7}. Given discrete data $\sigma_d= \left\{(z_k,c_{i,k}), i=0,1, z_k\in \mathcal{Z} \right\}^N_{k=1}$. Find a matrix-valued function $m(z|\sigma_d)$ which has the following properties:
   \begin{itemize}
       \item[(a)]$m(z|\sigma_d)$ is analytical in $\mathbb{C}\setminus  \left( \mathcal{Z} \cup \bar{\mathcal{Z}} \right)$;
       \item[(b)]$m(z|\sigma_d)=I + O(z^{-1}),	\quad   z \to  \infty $;
       \item[(c)]$m(z|\sigma_d)$ satisfies the following relations at each double pole  $z_k\in \mathcal{Z}$ and $\bar{z}_k \in \bar{\mathcal{Z}}$
   \begin{align}
    & \res_{z=z_k} m(z|\sigma_d)= \lim_{z\to z_k} \left[ m(z|\sigma_d) n_{0,k} + m'(z|\sigma_d) n_{1,k} \right], \label{rlres}\\
    & \res_{z=\bar{z}_k} m(z|\sigma_d)= \lim_{z\to \bar{z}_k} \left[ m(z|\sigma_d) \sigma_2 \overline{n_{0,k}} \sigma_2 + m'(z|\sigma_d) \sigma_2\overline{n_{1,k}}\sigma_2  \right],\\
    &\underset{z=z_k}{P_{-2}} m(z|\sigma_d) = \lim_{z\to z_k}  m(z|\sigma_d) n_{1,k},\label{rlp2}\\
   & \underset{z=\bar{z}_k}{P_{-2}} m(z|\sigma_d) = \lim_{z\to \bar{z}_k}  m(z|\sigma_d) \sigma_2 \overline{n_{1,k}}\sigma_2,
   \end{align}
   where
   \begin{equation}
        n_{0,k}=  \left(\begin{array}{ll}
            0 &  0\\
            \gamma_{0,k}(x,t) & 0
             \end{array}\right),\;
       n_{1,k}=  \left(\begin{array}{ll}
                0 &  0\\
                \gamma_{1,k}(x,t) & 0
                 \end{array}\right),
   \end{equation}
    with
    \begin{align}
    &\gamma_{0,k}(x,t)=c_{0,k}e^{2it \theta(z_k)},\\
    &\gamma_{1,k}(x,t)= c_{1,k} e^{2it \theta(z_k)},\\
    &c_{0,k} =A_k \left(B_k+2it \theta'(z_k)\right) ,\\
    &c_{1,k}= A_k.
    \end{align}
    \end{itemize}

    \begin{proposition} Given discrete data $\sigma_d= \left\{(z_k,c_{i,k}), i=0,1, z_k\in \mathcal{Z} \right\}^N_{k=1}$, there exists a unique solution of RHP7 for each $(x,t)\in \mathbb{R}^2$.
    \begin{equation}
        q_{sol}(x,t;\sigma_d)=2i \lim_{z\to \infty} \left(zm(z|\sigma_d)\right)_{12}.
    \end{equation}
    \end{proposition}
    \begin{proof}
        The proof includes two parts. One is for the uniqueness, and the other is for the existence.
        The proof of uniqueness is relatively simple, here we only briefly introduce the steps and mainly prove the existence.\\
        Uniqueness: To prove the uniqueness of this solution, we first need to introduce a transformation to remove singularity of $m(z|\sigma_d)$ and then use Liouville's theorem to provide the uniqueness.
        Existence:
    We rewrite  $\res_{z=z_k} m(z|\sigma_d)$ and $\underset{z=z_k}{P_{-2}} m(z|\sigma_d)$ into the following form:
    \begin{equation}
        \res_{z=z_k} m(z|\sigma_d)=a^{(0)}(z_k)n_{0,k}+a^{(1)}(z_k) n_{1,k}=   \left(\begin{array}{ll}
            a^{(0)}_{12}(z_k)\gamma_{0,k} + a^{(1)}_{12}(z_k) \gamma_{1,k} &  0\\
            a^{(0)}_{22}(z_k)\gamma_{0,k} + a^{(1)}_{22}(z_k) \gamma_{1,k} & 0
             \end{array}\right) \triangleq  \left(\begin{array}{ll}
               \alpha_{1,k} &  0\\
               \beta_{1,k} & 0
                 \end{array}\right),
    \end{equation}
    \begin{equation}
        \underset{z=z_k}{P_{-2}} m(z|\sigma_d) = a^{(2)}(z_k)n_{1,k} = \left(\begin{array}{ll}
            a^{(2)}_{12}(z_k)\gamma_{1,k} &  0\\
            a^{(2)}_{22}(z_k)\gamma_{1,k} & 0
             \end{array}\right)\triangleq  \left(\begin{array}{ll}
                \alpha_{2,k} &  0\\
                \beta_{2,k} & 0
                  \end{array}\right).
    \end{equation}

    From the symmetry  $m(z|\sigma_d)=\sigma_2 \overline{m(\bar{z}|\sigma_d)} \sigma_2 $, we know
    \begin{equation}
      \res_{z=\bar{z}_k} m(z|\sigma_d) = \sigma_2 \left[\overline{a^{(0)}(z_k)n_{0,k}+a^{(1)}(z_k) n_{1,k}}      \right]\sigma_2 =\left(\begin{array}{ll}
        0 &  -\bar{\beta}_{1,k}\\
        0 & \bar{\alpha}_{1,k}
          \end{array}\right),
    \end{equation}
    \begin{equation}
        \underset{z=\bar{z}_k}{P_{-2}} m(z|\sigma_d)=  \sigma_2 \overline{a^{(2)}(z_k)n_{2,k}} \sigma_2 = \left(\begin{array}{ll}
            0 &  -\bar{\beta}_{2,k}\\
            0 & \bar{\alpha}_{2,k}
              \end{array}\right).
    \end{equation}

    Notice that when $r(z)=0$, $v(z)=I$. Therefore, the above RH problem for $m(z|\sigma_d)$ has the following solution
    \begin{equation}\label{Taymrl}
        \begin{split}
        m(z|\sigma_d) &= I + \sum^{N}_{k=1}  \Big[ \frac{1}{z-z_k}  \left(\begin{array}{ll}
            \alpha_{1,k} &  0\\
            \beta_{1,k} & 0
              \end{array}\right) +\frac{1}{(z-z_k)^2}  \left(\begin{array}{ll}
                \alpha_{2,k} &  0\\
                \beta_{2,k} & 0
                  \end{array}\right) \\
                 & + \frac{1}{z-\bar{z}_k}  \left(\begin{array}{ll}
                    0 &   -\bar{\beta}_{1,k}\\
                    0 & \bar{\alpha}_{1,k}
                      \end{array}\right) +\frac{1}{(z-\bar{z}_k)^2}  \left(\begin{array}{ll}
                        0 &  -\bar{\beta}_{2,k}\\
                        0 & \bar{\alpha}_{2,k}
                          \end{array}\right) \Big].
          \end{split}
    \end{equation}
    Substituting (\ref{Taymrl}) into (\ref{rlres}) and (\ref{rlp2}) respectively, we get the following linear equations after normalization
   \begin{align}
   & \alpha_{1,j}+ \sum_{k=1}^N \Bigg[ \gamma_{0,j} \big( \frac{\bar{\beta}_{1,k}}{z_j-\bar{z}_k}  + \frac{\bar{\beta}_{2,k}}{(z_j-\bar{z}_k)^2}\big) - \gamma_{1,j} \big( \frac{\bar{\beta}_{1,k}}{(z_j-\bar{z}_k)^2}  + \frac{2\bar{\beta}_{2,k}}{(z_j-\bar{z}_k)^3}\big) \Bigg] =0, \label{le1}\\
   & \bar{\beta}_{1,j}- \sum_{k=1}^N \Bigg[ \bar{\gamma}_{0,j} \big( \frac{\alpha_{1,k}}{\bar{z}_j-z_k}  + \frac{\alpha_{2,k}}{(\bar{z}_j-z_k)^2}\big) - \bar{\gamma}_{1,j} \big( \frac{\alpha_{1,k}}{(\bar{z}_j-z_k)^2}  + \frac{2\alpha_{2,k}}{(\bar{z}_j-z_k)^3}\big) \Bigg] =\bar{\gamma}_{0,j},\\
   & \alpha_{2,j}+ \sum_{k=1}^N  \Bigg[ \gamma_{2,j} \big(  \frac{\bar{\beta}_{1,k}}{z_j-\bar{z}_k}  + \frac{\bar{\beta}_{2,k}}{(z_j-\bar{z}_k)^2}\big) \Bigg]  =0,\\
   & \bar{\beta}_{2,j} - \sum_{k=1}^N \Bigg[\bar{\gamma}_{2,j} \big( \frac{\alpha_{1,k}}{\bar{z}_j-z_k}  + \frac{\alpha_{2,k}}{(\bar{z}_j-z_k)^2}\big) \Bigg] =\bar{\gamma}_{2,j} \label{le4}.
   \end{align}
   Next, we transform the above linear equations (\ref{le1})-(\ref{le4}) into matrix form. Let
   \begin{align*}
    &\alpha_1 = (\alpha_{1,1}, \cdots,\alpha_{1,N} )^T,\;\alpha_2 = (\alpha_{2,1}, \cdots,\alpha_{2,N} )^T,\;\\
    &\bar{\beta}_1 = (\bar{\beta}_{1,1}, \cdots,\bar{\beta}_{1,N} )^T,\;\bar{\beta}_2 = (\bar{\beta}_{2,1}, \cdots,\bar{\beta}_{2,N} )^T,\\
   & A =(a_{ij})_{N \times N},\; a_{ij}= \frac{\gamma_{0,i} }{z_i-\bar{z}_j}-  \frac{\gamma_{1,i} }{(z_i-\bar{z}_j)^2},\; i,j=1,\cdots,N,\\
   & B =(b_{ij})_{N \times N},\; b_{ij}= \frac{\gamma_{0,i} }{(z_i-\bar{z}_j)^2}-  \frac{2\gamma_{1,i} }{(z_i-\bar{z}_j)^3},\; i,j=1,\cdots,N,\\
   & C =(c_{ij})_{N \times N},\; c_{ij}= \frac{\gamma_{1,i} }{z_i-\bar{z}_j},\; i,j=1,\cdots,N,\\
   & D =(d_{ij})_{N \times N},\; d_{ij}= \frac{\gamma_{1,i} }{(z_i-\bar{z}_j)^2},\; i,j=1,\cdots,N.
\end{align*}
  Thus, the above linear equations (\ref{le1})-(\ref{le4}) are equivalent to the following partitioned matrix equation
  \begin{equation}\label{ble}
    \left(\begin{array}{llll}
        I_N &   0 & A & B\\
         0 & I_N & C & D \\
         - A^* & - B^* &I_N &0\\
         -C^* & -D^* & 0 & I_N\\
           \end{array}\right)     \left(\begin{array}{llll}
            \alpha_1\\
            \alpha_2 \\
             \beta_1\\
            \beta_2\\
               \end{array}\right) =  \left(\begin{array}{llll}
                0\\
                0 \\
                 \bar{\gamma}_0\\
                 \bar{\gamma}_1\\
                   \end{array}\right).
  \end{equation}
It is easy to prove that the coefficient matrix of the above equation is positive definite.  According to the Cramer's rule, the solution of the (\ref{ble}) exists and is unique.
\end{proof}
\begin{remark}
    RHP7 is a special case  in Zhang's paper \cite{hpnls} with $n_1=n_2=\cdots=n_N=2$, but the method of proof in his article is completely different from ours.
\end{remark}

 For convenience, let $\Delta\subseteqq\left\lbrace 1,2,\cdots,N\right\rbrace $, $\nabla= \Delta^c = \left\lbrace 1,2,\cdots,N\right\rbrace \setminus \Delta$ and define
 \begin{equation}
 a_\Delta(z)=\prod_{k\in\Delta }\left(\dfrac{z-z_k}{z-\bar{z}_k}\right)^2,\; a_\nabla(z)= \frac{s_{11}}{ a_\Delta(z)}\prod_{k\in\nabla }\left(\dfrac{z-z_k}{z-\bar{z}_k}\right)^2.
 \end{equation}
Then we make a transform
   \begin{equation}\label{renormalization}
    m^\Delta\left(z|\sigma_d^\Delta\right)=m\left(z|\sigma_d\right)a_\Delta(z)^{\sigma_3}.
    \end{equation}
As we can see from the above expression, the  transformation (\ref{renormalization})   splits the poles between the columns of $m^\Delta\left(z|\sigma_d^\Delta\right)$
according to the choice of $\Delta$, and it satisfies the following nonreflective RH problem.

\noindent\textbf{RHP8}. Given discrete data $\sigma_d^\Delta= \left\{\left(z_k,c^\Delta_{i,k}\right), i=0,1, z_k\in \mathcal{Z} \right\}^N_{k=1}$. Find a matrix-valued function $m^\Delta\left(z|\sigma_d^\Delta\right)$ which has the following properties:
\begin{itemize}
    \item[(a)]$m^\Delta\left(z|\sigma_d^\Delta\right)$ is analytical in $\mathbb{C}\setminus  \left( \mathcal{Z} \cup \bar{\mathcal{Z}} \right)$;
    \item[(b)]$m^\Delta\left(z|\sigma_d^\Delta\right)=I + O(z^{-1}),	\quad   z \to  \infty $;
    \item[(c)]$m^\Delta\left(z|\sigma_d^\Delta\right)$ has the following relations at discrete spectrum $ \mathcal{Z} \cup \bar{\mathcal{Z}}$
\begin{align}
 &\res_{z=z_k} m^\Delta\left(z|\sigma_d^\Delta\right)= \lim_{z\to z_k} \left[ m^\Delta\left(z|\sigma_d^\Delta\right) n^\Delta_{0,k} + \left(m^\Delta\right)'\left(z|\sigma_d^\Delta\right) n^\Delta_{1,k} \right], \label{rlresp}\\
 &\res_{z=\bar{z}_k} m^\Delta\left(z|\sigma_d^\Delta\right)= \lim_{z\to \bar{z}_k} \left[ m^\Delta\left(z|\sigma_d^\Delta\right) \sigma_2 \overline{n^\Delta_{0,k}} \sigma_2 + \left(m^\Delta\right)'\left(z|\sigma_d^\Delta\right)\sigma_2\overline{n^\Delta_{1,k}}\sigma_2  \right],\\
 &\underset{z=z_k}{P_{-2}} m^\Delta\left(z|\sigma_d^\Delta\right) = \lim_{z\to z_k}  m^\Delta\left(z|\sigma_d^\Delta\right)n^\Delta_{1,k},\label{rlp2p}\\
 &\underset{z=\bar{z}_k}{P_{-2}} m^\Delta\left(z|\sigma_d^\Delta\right) = \lim_{z\to \bar{z}_k}  m^\Delta\left(z|\sigma_d^\Delta\right)\sigma_2 \overline{n^\Delta_{1,k}}\sigma_2,
\end{align}
where
\begin{equation}
     n^\Delta_{0,k}= \begin{cases} \left(\begin{array}{ll}
         0 &  \gamma^\Delta_{0,k}(x,t)\\
         0 & 0
          \end{array}\right),\; & \text{ $k \in \Delta$,}\\
          \left(\begin{array}{ll}
            0 & 0\\
            \gamma^\Delta_{0,k}(x,t) & 0
             \end{array}\right),\; & \text{ $k \in \nabla$,}
            \end{cases}
 \end{equation}
 \begin{equation}
    n^\Delta_{1,k}=\begin{cases} \left(\begin{array}{ll}
        0 &  \gamma^\Delta_{1,k}(x,t)\\
        0 & 0
         \end{array}\right),\; & \text{ $k \in \Delta$,}\\
         \left(\begin{array}{ll}
           0 & 0\\
           \gamma^\Delta_{1,k}(x,t) & 0
            \end{array}\right),\; & \text{ $k \in \nabla$,}
           \end{cases}
\end{equation}
 with
 \begin{align}
    & \gamma^\Delta_{0,k}(x,t) =\begin{cases}
 c^\Delta_{0,k}e^{-2it \theta(z_k)},\; & \text{ $k \in \Delta$,}\\
 c^\Delta_{0,k}e^{2it \theta(z_k)},\; & \text{ $k \in  \nabla$,}
\end{cases}  \\
& c^\Delta_{0,k}  =\begin{cases}
-4A_k^{-1} a''_\Delta(z_k)^{-2} \left[ 2it \theta'(z_k) + B_k +\frac{2a'''_\Delta(z_k)}{3a''_\Delta(z_k)}     \right] ,\; & \text{ $k \in \Delta$,}\\
2 b_k s''_{11}(z_k)a''_\nabla(z_k)^{-2} \left[ 2it \theta'(z_k) + B_k -\frac{2a'''_\nabla(z_k)}{3a''_\nabla(z_k)}     \right] ,\; & \text{ $k \in \nabla$,}
\end{cases}\\
& \gamma^\Delta_{1,k}(x,t) =\begin{cases}
 c^\Delta_{1,k} e^{-2it \theta(z_k)},\; & \text{ $k \in \Delta$,}\\
 c^\Delta_{1,k} e^{2it \theta(z_k)},\; & \text{ $k \in  \nabla$,}
\end{cases} \\
& c^\Delta_{1,k} =\begin{cases}
    4A_k^{-1} a''_\Delta(z_k)^{-2} ,\; & \text{ $k \in \Delta$,}\\
    b_k s''_{11}(z_k)a''_\nabla(z_k)^{-2},\; & \text{ $k \in \nabla$.}
\end{cases}
 \end{align}
 \end{itemize}
The proof is similar to the proof of RHP3.

\begin{proposition}
    For nonreflective scattering data $\sigma_d^\Delta= \left\{\left(z_k,c^\Delta_{i,k}\right), i=0,1, z_k\in \mathcal{Z} \right\}^N_{k=1}$, RHP8 owns a unique solution and
\begin{equation}\label{qsolt}
    q_{sol}\left(x,t;\sigma_d^\Delta\right)= 2i \lim_{z\to \infty} \left(zm^\Delta\left(z|\sigma_d^\triangle\right) \right)_{12}  = 2i \lim_{z\to \infty} \left(zm\left(z|\sigma_d\right) \right)_{12}=    q_{sol}\left(x,t;\sigma_d\right).
\end{equation}
\end{proposition}
\begin{proof}
    Since the transformation (\ref{renormalization}) is explicit, from Proposition 3, we know this RHP8  has a unique solution.
    Using transformations (\ref{mTay}) and (\ref{renormalization}), we can obtain
    \begin{equation}
        m^\Delta\left(z|\sigma_d^\Delta\right)=I +\frac{1}{2iz} \left(\begin{array}{cc}
            -\int^{\infty}_{x} |q|^2 dx + 8 \Sigma_{k \in \Delta} \text{Im}z_k & q(x,t)\\
            \bar{q}(x,t) & \int^{\infty}_{x} |q|^2 dx -8 \Sigma_{k \in \Delta} \text{Im}z_k
            \end{array}\right)+o\left(z^{-1}\right).
    \end{equation}
    Hence, the formula (\ref{qsolt}) can be found.
\end{proof}

    In order to  establish the relationship between $m^{(out)}(z)$ and $m^\Delta\left(z|\sigma_d^\Delta\right)$, we take  $\Delta =\Delta^-_{z_0}$ and replace the scattering data $\sigma_d^\Delta$ with  scattering data
    \begin{align}\label{sigout}
    \sigma_d^{(out)}=\left\lbrace \left(z_k,\tilde{c}_{i,k}\right) ,  \tilde{c}_{i,k}=c_{i,k}\delta(z_k)^2, \; i=0,1 \right\rbrace _{k=1}^N.
    \end{align}
    Notice that the conditions defining $m^{(out)}(z)$ are identical to those defining $m^{\Delta^-_{z_0}}\left(z| \sigma_d^{(out)}\right)$, we can draw a conclusion.
    \begin{corollary}
        There exists a unique solution  for the RHP5. Moreover,
       \begin{equation}\label{moutmz0}
       m^{(out)}(z)=m^{\Delta^-_{z_0}}\left(z| \sigma_d^{(out)}\right),
       \end{equation}
       where the scattering data  $\sigma_d^{(out)}$ is given by (\ref{sigout}).
       In addition, the corresponding N-soliton solution satisfies
       \begin{equation*}
        q_{sol}\left(x,t;\sigma_d^{(out)}\right)=q_{sol}\left(x,t;\sigma_d^{\Delta^-_{z_0}}\right).
       \end{equation*}
    \end{corollary}

    Next, we consider the large $z$ behavior of the above solutions.

\begin{figure}
    \begin{center}
   \begin{tikzpicture}
           \draw[yellow!20, fill=yellow!20] (0.6,0)--(1.5,2.5)--(-1.0,2.5)--(-0.5,0)--(-1.4,-2.5)--(1.0,-2.5)--(0.6,0);
           \draw[thick,->](-3.5,0)--(3.5,0)node[right]{$x$};
           \draw[thick,->](-1.5,-3)--(-1.5,3)node[right]{$t$};
           \draw[](0.6,0)--(1.5,2.5)node[right]{$x=v_2t+x_2$};
           \draw[](-0.5,0)--(-1.4,-2.5)node[left]{$x=v_2t+x_1$};
           \draw[](-0.5,0)--(-1.0,2.5)node[left]{$x=v_1t+x_1$};
           \draw[](0.6,0)--(1.0,-2.5)node[right]{$x=v_1t+x_2$};
           \draw[](0,0)--(-0.8,0);
           \draw[](-0.8,0)--(-1.8,0);
           \draw[](0,0)--(0.8,0);
           \draw[](0.8,0)--(1.8,0);
           \coordinate (A) at (0.9,0);
           \coordinate (B) at (-0.8,0);
           \coordinate (I) at (0,2);
           \fill (A)  node[below] {$x_2$};
           \fill (B)  node[below] {$x_1$};
           \fill (I) circle (0pt) node[below] {$\mathcal{S}$};
           \label{figzero}
           \end{tikzpicture}
       \caption{ Space-time  $\mathcal{S}(x_1,x_2,v_1,v_2)$ }
       \label{figC(I)}
   \end{center}	
   \end{figure}
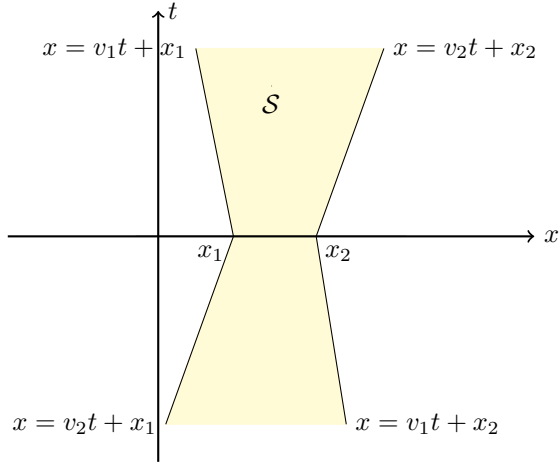

\begin{proposition}
 Given discrete scattering data $\sigma_d=\left\{ \left(z_k,c_{i,k}\right),i=0,1, z_k\in \mathcal{Z} \right\}_{k=1}^N$, pairs of points $y_1, y_2$ with $y_1\leq y_2\in \mathbb{R}$ and velocities $v_1, v_2$ with $v_1\leq v_2 \in \mathbb{R}$,  we  define the  cone
 \begin{equation}\label{coneS}
 \mathcal{S}(y_1,y_2,v_1,v_2):=\left\lbrace (y,t)\in R^2|y=y_0+vt, \ y_0\in[y_1,y_2]\text{, }v\in[v_1,v_2]\right\rbrace.
 \end{equation}
See Figure \ref{figC(I)}.

 Take $I=[-v_2/2,-v_1/2]$. Then, when $t\to \infty$ with $(y,t)\in \mathcal{S}(y_1,y_2,v_1,v_2)$, we have
 \begin{equation}
 m^{\Delta^-_{z_0}}\left(z| \sigma_d\right) = \left( I+\mathcal{O}(e^{-2\mu|t|})\right) m^{\Delta^-_{z_0}(I)}\left(z;y,t|\sigma_d^-(I)\right)
\end{equation}

where
\begin{align}
&\sigma_d^-(I)=\left\{ \left(z_k,c_{i,k}(I)\right), i=0,1, z_k\in  \Delta_{z_0}^-(I) \right\}_{k=1}^N,\\
 &   \mu=\mu(I)=\underset{z_k\in \mathcal{Z}\setminus \mathcal{Z}(I) }{\min}\left\lbrace \text{\rm{Im}}(z_k)\text{\rm{dist}}( \text{\rm{Re}}{(z_k)},I) \right\rbrace=\underset{z_k\in \mathcal{Z}\setminus \mathcal{Z}(I),i=1,2}{\min}\left\lbrace 2\text{\rm{Im}}(z_k)|v_i-v_{z_k}|\right\rbrace.
\end{align}
An example has been given in Figure \ref{figS}.
\end{proposition}

\begin{figure}
    \begin{center}
   \begin{tikzpicture}
           \filldraw[yellow!20,line width=2] (-1,-2.4) rectangle (1,2.4);
           \draw[thick,->](-3.5,0)--(3.5,0)node[right]{Re$z$};
           \draw[](1,-2.5)--(1,2.5);
           \node[below] at (-1.5,0) {$-\frac{v_2}{2}$};
           \draw[]    (-1,-2.5)--(-1,2.5);
           \node[below] at (1.5,0)  {$-\frac{v_1}{2}$};
             \node[red]    at (0.3,0) {$\ast$};
           \node    at (0.3,-0.3)  {$z_0$};
           \coordinate (A) at (0.5,1.2);
           \coordinate (B) at (0.6,-1.2);
           \coordinate (C) at (1.5,1.6);
           \coordinate (D) at (1.6,-1.6);
           \coordinate (E) at (-2.2,0.8);
           \coordinate (F) at (-2.2,-0.8);
           \coordinate (G) at (-0.3,1.9);
           \coordinate (H) at (-0.3,-1.9);
           \coordinate (I) at (2.5,1);
           \coordinate (J) at (2.5,-1);
           \fill (A) circle (1pt) node[right] {$z_1$};
           \fill (B) circle (1pt) node[right] {$\bar{z}_1$};
           \fill (G) circle (1pt) node[left] {$z_3$};
           \fill (H) circle (1pt) node[left] {$\bar{z}_3$};
           \fill (C) circle (1pt) node[right] {$z_2$};
           \fill (D) circle (1pt) node[right] {$\bar{z}_2$};
           \fill (E) circle (1pt) node[right] {$z_4$};
           \fill (F) circle (1pt) node[right] {$\bar{z}_4$};
           \fill (I) circle (1pt) node[right] {$z_5$};
           \fill (J) circle (1pt) node[right] {$\bar{z}_5$};
           \draw [red] (C) circle [radius=0.2];
           \draw [red,  -> ]  (1.7,1.6) to  [out=90, in=0]  (1.5,1.8);
           \draw [red] (D) circle [radius=0.2];
            \draw [red,  -> ]  (1.8,-1.6) to  [out=90, in=0]  (1.6,-1.4);	
   \draw [red] (I) circle [radius=0.2];
           \draw [red,  -> ]  (2.7,1) to  [out=90, in=0]  (2.5, 1.2);	
   \draw [red] (J) circle [radius=0.2];
           \draw [red,  -> ]  (2.7,-1) to  [out=90, in=0]  (2.5, -0.8);	
   \draw [red] (E) circle [radius=0.2];
           \draw [red,  -> ]  (-2,0.8) to  [out=90, in=0]  (-2.2,1);
   \draw [red] (F) circle [radius=0.2];
           \draw [red,  -> ]  (-2,-0.8) to  [out=90, in=0]  (-2.2,-0.6);
   \end{tikzpicture}
   \caption{ Fix $v_1$ and $v_2$ so that $v_1<v_2$. $I=[-v_2/2,-v_1/2]$.  For example, the original data has five pairs of discrete spectrum, but inside
   cone $\mathcal{S}(x_1,x_2,v_1,v_2)$,  the solution is asymptotically described by a two-soliton solution  $q_{sol}(x,t;\sigma_d(I))$ with discrete spectrum in $\mathcal{Z}(I)=\left\{z_1, z_3\right\}$.}
       \label{figS}
    \end{center}
   \end{figure}
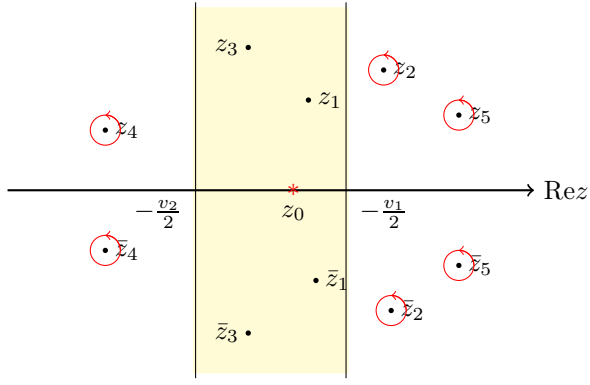

\begin{proof}
    Let
    $$\Delta^-(I)=\left\lbrace k||z_k|<-v_2/2\right\rbrace,  \ \
    \Delta^+(I)=\left\lbrace k||z_k| >-v_1/2\right\rbrace.$$
    For $t>0$, $(x,t)\in \mathcal{S}(x_1,x_2,v_1,v_2)$, we have
    $$ -v_2/2 < z_0+x_0/(2t)<-v_1/2,$$
    and as $t \to \infty$, $x_0/(2t)\to 0 $, $-v_2/2<z_0 < -v_1/2.$
    By the residue condition and the coefficient of negative second power of Laurent  expansion, it is easy to calculate that
    \begin{equation}\label{nest}
        \parallel n_{i,k}^{\Delta^\pm(I)}\parallel= \begin{cases}
            \mathcal{O}(1), & z_k \in \mathcal{Z}(I),\\
            \mathcal{O}\left(e^{-4\mu|t|}\right), & z_k \in \mathcal{Z} \setminus \mathcal{Z}(I).
        \end{cases}
    \end{equation}
    For each $z_k\in \mathcal{Z}\setminus \mathcal{Z}(I)$, we introduce small disks $D_k$ whose radius are sufficiently small that they are non-overlapping.
    We define a function
    \begin{equation}
        \Upsilon(z)= \begin{cases}
            I - \frac{1}{z-z_k}n^{\Delta^-_{z_0}}_{0,k} -\frac{1}{(z-z_k)^2}n^{\Delta^-_{z_0}}_{2,k}, & z\in D_k,\\
            I -  \sigma_2 \left[ \frac{1}{z-\bar{z}_k}\bar{n}^{\Delta^-_{z_0}}_{0,k} -\frac{1}{(z-\bar{z}_k)^2}\bar{n}^{\Delta^-_{z_0}}_{2,k} \right] \sigma_2, & z\in \bar{D}_k,\\
            I, & \text{elsewhere}.
        \end{cases}
    \end{equation}
    Then we introduce a transformation
    \begin{equation}\label{nthm}
       \hat{m}^{\Delta^-_{z_0}}\left(z| \sigma_d^{\Delta^-_{z_0}}\right)=  {m}^{\Delta^-_{z_0}}\left(z| \sigma_d^{\Delta^-_{z_0}}\right) \Upsilon(z).
    \end{equation}
    Furthermore, $ \hat{m}^{\Delta^-_{z_0}}\left(z| \sigma_d^{\Delta^-_{z_0}}\right)$ has jumps across each boundary of $D_k$ and $\bar{D}_k$,
    \begin{equation}
        \hat{m}_+^{\Delta^-_{z_0}}\left(z| \sigma_d^{\Delta^-_{z_0}}\right)=\hat{m}_-^{\Delta^-_{z_0}}\left(z| \sigma_d^{\Delta^-_{z_0}}\right) \hat{v}(z), \quad z \in \partial D_k \cup \partial \bar{D}_k,
    \end{equation}
    with
    \begin{equation}
        \parallel \hat{v}-I \parallel = \mathcal{O}\left(e^{-4\mu|t|}\right)
    \end{equation}
    which can be given by the formula (\ref{nest}).

    Take $\Delta = \Delta_{z_0}^-(I)$, then $m^{ \Delta_{z_0}^-(I)}\left(z|\sigma^-_d(I)\right)$ has the same poles as  $ \hat{m}^{\Delta^-_{z_0}}\left(z| \sigma_d^{\Delta^-_{z_0}}\right)$ with the same residue conditions and the coefficient of negative second power of Laurent  expansion.
    Hence,
    \begin{equation}\label{varepsilon}
         \varepsilon(z)= \hat{m}^{\Delta^-_{z_0}}\left(z| \sigma_d^{\Delta^-_{z_0}}\right) \left[m^{ \Delta_{z_0}^-(I)}\left(z|\sigma^-_d(I)\right) \right]^{-1}
        \end{equation}
    has no poles but satisfies $\varepsilon_+(z)=\varepsilon_-(z)v_\varepsilon(z)$ with $ \parallel v_\varepsilon- I \parallel= \mathcal{O}\left(e^{-4\mu|t|}\right)$.
     From the theory of small-norm Riemann-Hilbert problems, we know $\varepsilon(z)=I + \mathcal{O}\left(e^{-4\mu|t|}\right)$ as $|t|\to \infty$.
     Finally, by  equations (\ref{nthm})  and (\ref{varepsilon}), we obtain
     \begin{equation}
        m^{\Delta^-_{z_0}}(z| \sigma_d) = \left( I+\mathcal{O}(e^{-2\mu|t|})\right) m^{\triangle^-_{z_0}(I)}\left(z;y,t|\sigma^-_d(I)\right).
     \end{equation}
\end{proof}

    Therefore, we can obtain the following corollary.
\begin{corollary}
    Assume that $q_{sol}\left(x,t;\sigma_d^{\Delta_{z_0}^-}\right)$ is the $N$-soliton solution of the NLS equation with scattering data $\sigma_d^{\Delta^-_{z_0}}= \left\{\left(z_k,c^{\Delta^-_{z_0}}_{i,k}\right), i=0,1, k\in \Delta^-_{z_0} \right\}^N_{k=1}$.
    Then, as $(x,t)\in \mathcal{S}(x_1,x_2,v_1,v_2)$, $t \to \infty$,
    \begin{equation}
        q_{sol}\left(x,t;\sigma_d^{(out)}\right)=q_{sol}\left(x,t;\sigma_d^{\Delta_{z_0}^-}\right)= q_{sol}\left(x,t;\sigma_d^-(I)\right) + \mathcal{O}\left(e^{-4\mu|t|}\right),
    \end{equation}
    where $ q_{sol}\left(x,t;\sigma_d^-(I)\right)$ is the $N(I)$-soliton solution of the NLS equation with the scattering data $\sigma_d^-(I)$.
\end{corollary}

\subsubsection{The construction of local model}
   At the beginning of the construction of local model near the saddle point, we consider the jump matrix in the interior of the region.
\begin{proposition}
    \begin{equation}
    \parallel v^{(2)}-I \parallel_{L^\infty(\Sigma^{(2)})}= \begin{cases}
        \mathcal{O}\left(|z- z_0|^{-1} t^{-1/2} \right), & z \in \Sigma^{(2)} \cap \epsilon_{ z_0},\\
        \mathcal{O}\left(e^{-t \rho^2/2}\right), & z \in \Sigma^{(2)} \setminus \epsilon_{z_0}.
    \end{cases}
    \end{equation}
\end{proposition}
\begin{proof}
    We prove the  above proposition for the case $z\in \Sigma_1$, and other cases can be shown in a similar way.
    The jump line is $z-z_0=|z-z_0|e^{i \pi /4}$ at $\Sigma_1$ and
    \begin{equation}
        \theta =(z-z_0)^2-z_0^2=i|z-z_0|^2-z_0^2.
    \end{equation}
    Using (\ref{R}) and (\ref{v2}), we obtain
    \begin{equation}
        |R_1 e^{2it \theta(z)}| \le \left(\frac{1}{2}c_1+c_2 \left<\text{Re}z\right>^{-1/2}\right) e^{-2t |z-z_0|^2}
    \end{equation}
    where $\left<\text{Re}z\right>^{-1/2} \le c$.
    In the interior of $\epsilon_{z_0}$, $m^{(2)}_{RHP}$ has no pole and
    \begin{equation}
        \parallel v^{(2)}-I \parallel_{L^\infty(\Sigma^{(2)})} \le c |z- z_0|^{-1} t^{-1/2},\quad z \in \Sigma^{(2)} \cap \epsilon_{ z_0}.
    \end{equation}
   Thus, it is  clear that the jump $v^{(2)}$ is point-wise bounded, but  not uniformly decayed to the identity matrix.
    Additionally, as $z \in  \Sigma^{(2)} \cap \left\{  |z-z_0| \ge \rho/2   \right\}$,
    \begin{equation}
        \parallel v^{(2)}-I \parallel_{L^\infty(\Sigma^{(2)})} \le c e^{-t \rho^2/2}.
    \end{equation}
\end{proof}

   In order to achieve a uniformly small jump Riemann-Hilbert problem for the function $E(z)$ defined by (\ref{m2rhp}),  we establish a local model  $m^{(z_0)}$ which matches $m_{RHP}^{(2)}$ on $\Sigma^{(2)} \cap  \epsilon_{z_0}$.
   For this reason, the translation scale transformation is defined by
   \begin{equation}
    \lambda= \lambda(z)=2 \sqrt{t}(z-z_0).
   \end{equation}
   Notice that if we take $r_0=r(z_0)T_0(z_0)^{-2}e^{2i(\nu(z_0)\rm{log}(2\sqrt{t})-t z_0^2)}$, the jump of $m^{(2)}_{RHP}$ is in accordance with  that of the parabolic cylinder model problem $m^{(pc)}$, which satisfies the following RH problem, see more details in \cite{Its}.

   \noindent\textbf{RHP9}.  Fix $r_0 \in \mathbb{R}$, find an analytic function $m^{(pc)}(\lambda,r_0)$ such that
   \begin{itemize}
       \item[(a)]$m^{(pc)}(\lambda,r_0)$ is analytic in $\mathbb{C}\setminus \Sigma^{(2)}$.
       \item[(b)] $m^{(pc)}(\lambda,r_0)$ has continuous boundary value $m^{(pc)}_\pm(\lambda,r_0)$ on $\Sigma^{(2)}$
       \begin{equation}
        m^{(pc)}_+(\lambda,r_0)=m^{(pc)}_-(\lambda,r_0) v^{(pc)}(\lambda,r_0), \quad \zeta \in \Sigma^{(2)},
       \end{equation}
       where \begin{equation}
        v^{(pc)}(\lambda,r_0)= \begin{cases}
            \left( \begin{array}{ll}
                1 & 0 \\
              r_0 \lambda^{-2i\nu}e^{i \lambda^2/2} & 1
            \end{array} \right),  & \lambda=\Sigma_1,\\
            \left( \begin{array}{ll}
                1 & \frac{\bar{r}_0}{1+|r_0|^2} \lambda^{2i \nu} e^{-i \lambda^2/2} \\
              0 & 1
            \end{array} \right),  & \lambda=\Sigma_2,\\
            \left( \begin{array}{ll}
                1 & 0 \\
                \frac{r_0}{1+|r_0|^2}  \lambda^{-2i\nu}e^{i \lambda^2/2} & 1
            \end{array} \right),  & \lambda=\Sigma_3,\\
            \left( \begin{array}{ll}
                1 & \bar{r}_0\lambda^{2i \nu} e^{-i \lambda^2/2} \\
              0 & 1
            \end{array} \right),  & \lambda=\Sigma_4.
        \end{cases}
       \end{equation}
       See Figure \ref{upcomplex}.

       \begin{figure}
        \begin{center}
        \begin{tikzpicture}
        \draw[dashed] (-3,0)--(3,0);
        \draw[dashed]  [   ](-3,0)--(-1.5,0);
        \draw[dashed]  [   ](0,0)--(1.5,0);
        \draw[red,thick ](-2,-2)--(2,2);
        \draw[red,thick,-> ](-2,-2)--(-1,-1);
        \draw[red,thick, -> ](0,0)--(1,1);
        \draw[red,thick ](-2,2)--(2,-2);
        \draw[red,thick,-> ](-2,2)--(-1,1);
        \draw[red,thick, -> ](0,0)--(1,-1);
        \node  [below]  at (2.3,2 ) {$\Sigma_1$};
        \node  [below]  at (-2.3,2 ) {$\Sigma_2$};
        \node  [below]  at (-2.2 ,-1.6) {$\Sigma_3$};
        \node  [below]  at (1.7,-1.7) {$\Sigma_4$};
        \node  [below]  at (0,-0.2) {$0$};

        \node    at (0,0)  {$\cdot$};
        \node [thick ] [below]  at (3.3, 1.5) {\tiny $ \left(\begin{array}{cc} 1&0\\ r_0 \lambda^{-2i\nu } e^{ i\lambda^{2}/2}  &1\end{array}  \right)  $};
        \node [thick ] [below]  at (-3.8,-0.3) {\tiny $ \left(\begin{array}{cc} 1&0\\ \frac{  r_0}{1+|r_0|^{2}}\lambda^{ -i\nu } e^{ i\lambda^{2}/2}  &1\end{array}  \right) $};
        \node [thick ] [below]  at (-3.6,1.5) {\tiny $\left(\begin{array}{cc} 1& \frac{\bar r_0}{1+|r_0|^{2}} \lambda^{2 i\nu } e^{- i\lambda^{2}/2} \\ 0&1\end{array}  \right)$};
        \node [thick ] [below]  at (3.5,-0.3) {\tiny $\left(\begin{array}{cc} 1& \bar r_0 \lambda^{2 i\nu } e^{- i\lambda^{2}/2} \\ 0&1\end{array}  \right)$};
        \end{tikzpicture}
        \end{center}
        \caption{Jump matrix  $v^{(pc)}$.}
        \label{upcomplex}
        \end{figure}
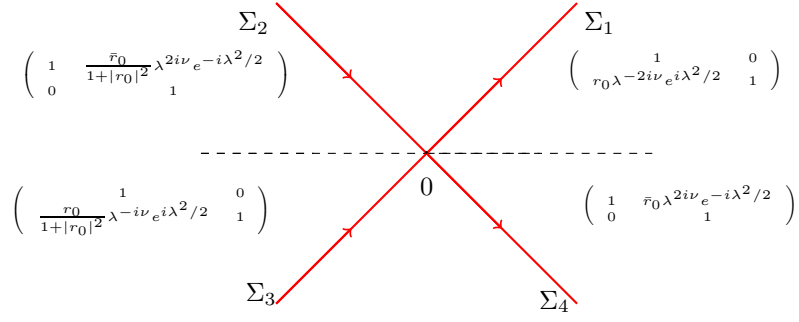

       \item[(c)]As $\lambda \to \infty$, $m^{(pc)}(\lambda,r_0)= I +\frac{m_{1}^{(pc)}(r_0)}{\lambda}+\mathcal{O}(\lambda^{-2}).$
   \end{itemize}
   Moreover, the asymptotic behavior of $m^{(pc)}(\lambda,r_0)$ has been verified in the paper \cite{sdmRHp}, which is
   \begin{equation}
    m^{(pc)}(\lambda,r_0)= I + \frac{1}{\lambda} \left(\begin{array}{ll}
        0 & -i \beta_{12}(r_0) \\
        i \beta_{21}(r_0) & 0
      \end{array} \right)  + \mathcal{O}\left(\lambda^{-2}\right),
   \end{equation}
   where
   \begin{equation}
    \beta_{12}(r_0)= \frac{\sqrt{2\pi}e^{i \pi /4} e^{- \pi \nu/2}}{r_0 \Gamma(-i \nu)},\quad  \beta_{21}(r_0)= \frac{\sqrt{2\pi}e^{-i \pi /4} e^{- \pi \nu/2}}{\bar{r}_0 \Gamma(i \nu)}= \frac{\nu}{\beta_{12}(r_0)}.
   \end{equation}
    Therefore, it is convenient to define $m^{(z_0)}(z)$, which is given by (\ref{m2rhp}), as follows
    \begin{equation}\label{mz0moutmpc}
        m^{(z_0)}(z)=m^{(out)}(z) m^{(pc)}(\lambda,r_0).
    \end{equation}
    Further, $m^{(z_0)}(z)$ is a bounded function in $\epsilon_{z_0}$ and fulfills the jump condition $v^{(2)}(z)$ of $m^{(2)}_{RHP}(z)$.

\subsection{The small-norm RH problem for $E(z)$}
 In this section, we deal with the error function $E(z)$. By  the definition (\ref{moutmz0}) and (\ref{mz0moutmpc}), it is obvious that $E(z)$ meets  the RH problem   as below.

 \noindent\textbf{RHP10}.  Find a holomorphic function $E(z)$ such that
 \begin{itemize}
    \item[(a)]  $E(z)$ is analytical  in $\mathbb{C}\setminus  \Sigma^{(E)} $, where
    $\Sigma^{(E)}= \partial \epsilon_{z_0}\cup \left(\Sigma^{(2)}\setminus \epsilon_{z_0}\right)$, see Figure \ref{figE}.

    \begin{figure}
        \begin{center}
       \begin{tikzpicture}
       \draw [ thick](-2.8,-2.8)--(2.8,2.8);
       \draw [thick, ->](0,0)--(2,2);
       \draw [thick, ->](-2.8,-2.8)--(-2,-2);
       \draw [ thick ](-2.8,2.8)--(2.8,-2.8);
       \draw [thick, ->](-2.8,2.8)--(-2,2);
       \draw [thick, ->](0,0)--(2,-2);
       \path [fill=white]  (-1.2,0) to  [out=-90, in=180]  (0, -1.2) to [out=0, in=-90]  (1.2,0);
       \path [fill=white]  (-1.2,0) to  [out=90, in=180]  (0,1.2) to [out=0, in=90]  (1.2,0);
       \draw [ thick,red] (0,0) circle [radius=1.2];
       \draw [ thick,red,->]  (-1.2,0) to  [out=90, in=180]  (0,1.2);
       \node [thick] [above]  at (1.7,0){\footnotesize $ \partial\epsilon_{z_0}$};
       \node [thick] [above]  at (1.5,2){\footnotesize $ \Sigma^{(2)}\setminus  \epsilon_{z_0}$};
       \end{tikzpicture}
       \end{center}
       \caption{  \text{Jump contour} $\Sigma^{(E)}=\partial \epsilon_{z_0}\cup \left(\Sigma^{(2)}\backslash \epsilon_{z_0}\right)$}
       \label{figE}
       \end{figure}
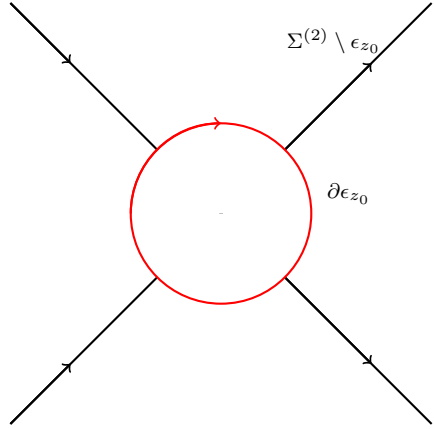

    \item[(b)]  For $z\in  \Sigma^{(E)}$, $E(z)$ has continuous boundary values $E_\pm(z)$ which satisfy
       $$E_+(z)=E_-(z)v^{(E)}(z), $$
 where
   \begin{equation}\label{deVE}
   v^{(E)}(z)=\left\{\begin{array}{llll}
   m^{(out)}(z)v^{(2)}(z)m^{(out)}(z)^{-1}, & z\in \Sigma^{(2)}\setminus \epsilon_{z_0},\\[4pt]
   m^{(out)}(z)m^{(pc)}\left(\lambda,r_0\right)m^{(out)}(z)^{-1},  & z\in \partial \epsilon_{z_0}.
   \end{array}\right.
\end{equation}
    \item[(c)] For $z\in \infty$, $E(z)=I+\mathcal{O}\left(z^{-1}\right)$.
 \end{itemize}
     Utilizing Proposition 6, the formula (\ref{deVE}) and the boundedness of $ m^{(out)}(z)$, we can obtain
     \begin{equation}\label{vE-I}
        \left|v^{(E)}(z)-I\right|=\left\{\begin{array}{llll}
        \mathcal{O}\left(e^{-t \rho^2/2}\right),  & z\in \Sigma^{(2)}\setminus \epsilon_{z_0},\\[6pt]
        \mathcal{O}\left(t^{-1/2}\right),   & z\in  \partial \epsilon_{z_0}.
        \end{array}\right.
        \end{equation}

  \begin{proposition}
    RHP10 has a unique solution.
  \end{proposition}
  \begin{proof}
    According to Beal-Cofiman theorem, the solution of  RHP10 can be constructed by
    \begin{equation}
        E(z)=I+\frac{1}{2\pi i}\int_{\Sigma^{(E)}}\dfrac{\mu_E(s) \left(v^{(E)}(s)-I\right)}{s-z}ds,\label{Ez}
        \end{equation}
    where $\mu_E \in L^2\left(\Sigma^{(E)}\right)$ satisfies
    \begin{equation}
        \left(1-C_{w_E}\right)\mu_E=I
    \end{equation}
    with  $C_{w_E}$ being  a integral operator defined by
    \begin{equation}\label{Cwe}
          C_{w_E}(f)(z)=C_-\left( f\left(v^{(E)}(z) -I\right)\right) ,
    \end{equation}
    where $C_-$ is the Cauchy projection operator
    \begin{equation}\label{Cauchy}
        C_-(f)(s)=\lim_{z'\to  z \in \Sigma^{(E)}}\frac{1}{2\pi i}\int_{\Sigma^{(E)}}\dfrac{f(s)}{s-z'}ds.
    \end{equation}

    Using the above formulas (\ref{Cwe}) and (\ref{Cauchy}), we get
    \begin{equation}
        \left\| C_{w_E} \right\|_{L^2(\Sigma^{(E)})}      \leq \left\| C_- \right\|_{L^2(\Sigma^{(E)})} \left\| v^{(E)}-I \right\|_{L^\infty(\Sigma^{(E)})} \lesssim \mathcal{O}\left(t^{-1/2}\right).
    \end{equation}
    This means that  $1-C_{w_E}$ is invertible.  Subsequently,  $\mu_E$ and the solution of the RHP10  exist and are unique.
  \end{proof}

   By (\ref{vE-I}) and the process of proof of Proposition 7, it is straightforward to obtain the following corollary.
\begin{corollary}
    \begin{align}
       & \left\| \left< \cdot  \right>   \left(v^{(E)} -I\right)\right\|_{L^p}= \mathcal{O}\left(t^{-1/2}\right), \quad p \in [1,\infty), \; k \ge 0,\\
       & \left\| \mu_E - I \right\|_{L^2(\Sigma^{(E)})} =  \mathcal{O}\left(t^{-1/2}\right).
    \end{align}
\end{corollary}

  Next, we consider the asymptotic expansion of $E(z)$
  \begin{equation}
    E(z)= I + \frac{E_1}{z} + \mathcal{O}\left(z^{-2}\right),
  \end{equation}
  where
  \begin{equation}\label{E1xt}
    E_1 = - \frac{1}{2 \pi i } \int_{\Sigma^{(E)}} \mu_E(s) \left(v^{(E)} -I\right)ds.
  \end{equation}
  \begin{proposition}
    \begin{equation}
        E_1(x,t)= \frac{1}{2i \sqrt{t}} m^{(out)}(z_0) m^{(pc)}_1(r_0) m^{(out)}(z_0)^{-1} + \mathcal{O}\left(t^{-1}\right).
    \end{equation}
  \end{proposition}
  \begin{proof}
    Rewrite formula (\ref{E1xt}) as
    \begin{equation}
        E_1 = -\frac{1}{2 \pi i } \oint_{\partial \epsilon_{z_0}}\left(v^{(E)}-I\right)ds -  -\frac{1}{2 \pi i } \int_{\Sigma^{(E)} \setminus \epsilon_{z_0}} \left(v^{(E)}-I\right)ds -\frac{1}{2 \pi i } \int_{\Sigma^{(E)} }\left(\mu_E(s)-I\right) \left(v^{(E)}-I\right)ds.
    \end{equation}
    Then, the conclusion can be obtained by estimating the three integrals respectively.
  \end{proof}

\subsection{The pure $\bar{\partial}$-Problem  and its asymptotic behaviors }
In this section,  we acquire a  pure $\bar{\partial}$-problem after removing the $\bar{\partial}$ component of $m^{(2)}$.
We define
\begin{equation} \label{transm3}
    m^{(3)}(z):= m^{(2)}(z)m^{(2)}_{RHP}(z)^{-1}.
\end{equation}
Then, $m^{(3)}(z)$ is continuous and has no jump in $\mathbb{C}$, which satisfies a pure $\bar{\partial}$-problem.

\noindent\textbf{RHP11}.  Find a matrix-valued function $m^{(3)}(z)$ with the following properties.
 \begin{itemize}
    \item[(a)]  $m^{(3)}(z)$ is continuous in $\mathbb{C} \setminus  \left(\mathbb{R}\cup \Sigma^{(2)}\right)$;
    \item[(b)]  For $z\in \mathbb{C}$, $\bar{\partial}m^{(3)}(z)=m^{(3)}(z)w^{(3)}(z)$ where
    \begin{equation}
        w^{(3)}(z)=m^{(2)}_{RHP}(z) \bar{\partial} \mathcal{R}^{(2)} m^{(2)}_{RHP}(z)^{-1};
    \end{equation}
    \item[(c)]  For $z\in \infty$, $m^{(3)}(z)=I + \mathcal{O}\left(z^{-1}\right)$.
 \end{itemize}

 \begin{proof}
    By the definition of $m^{(3)}(z)$, it is obvious that $m^{(3}(z)$ is continuous in $\mathbb{C} \setminus \Sigma^{(2)}$ and satisfies the condition (b).
    At the same time, using formulas (\ref{v2}) and (\ref{transm3}), we can find that
   \begin{equation*}
      m^{(3)}_+(z)=m_-^{(3)}(z), \quad z \in \Sigma^{(2)}.
   \end{equation*}

    From  the property (e) of RHP3, we prove the residue condition for the case $k \in \Delta^+_{z_0}$, $\text{Im}z_k>0$, because the proofs for the other cases are similar.
   \begin{align}
    &\res_{z=z_k} m^{(2)}(z) = \left( \gamma^{(2)}_{0,k}   m^{(2)}_2(z_k), 0 \right) + \left( \gamma^{(2)}_{1,k}   \left(m^{(2)}_2\right)'(z_k), 0 \right) = \lim_{z\to z_k}\left(m^{(2)}n^{(2)}_{0,k}+\left(m^{(2)}\right)'n^{(2)}_{1,k}\right),\label{m2res}\\
    &\underset{z=z_k}{P_{-2}} m^{(2)}(z) =  \left( \gamma^{(2)}_{1,k}   m^{(2)}_2(z_k), 0 \right)=\lim_{z\to z_k}\left(m^{(2)}n^{(2)}_{1,k}\right)\label{m2p-2},
   \end{align}
    where
    \begin{align}
       & \gamma^{(2)}_{0,k} = A_k T^{-2}(z_k) \left[ B_k+2it \theta'(z_k) -\frac{2 T'(z_k)}{ T(z_k)}  \right]e^{2it \theta(z_k)}, \\
       & \gamma^{(2)}_{1,k}=A_k T^{-2}(z_k) e^{2it \theta(z_k)},\\
       & n^{(2)}_{0,k} = \left(\begin{array}{ll}
        0 & 0\\
        \gamma^{(2)}_{0,k} & 0
       \end{array} \right), \quad n^{(2)}_{1,k} = \left(\begin{array}{ll}
        0 & 0\\
        \gamma^{(2)}_{1,k} & 0
       \end{array} \right).
    \end{align}
    Since $\left(n^{(2)}_{i,k}\right)^2=0 $($i=0,1$), $n^{(2)}_{i,k}$ is the nilpotent matrix.

    As $z_k$ is the  second-order pole of $m^{(2)}(z)$, $m^{(2)}(z)$ has the  Laurent expansion with the following form
    \begin{equation}\label{m2laurent}
        m^{(2)}(z)=   \frac{\underset{z=z_k}{P_{-2}} m(z|\sigma_d)      }{(z-z_k)^2} +\frac{\underset{z=z_k}{ \rm{Res}} m(z|\sigma_d)}{z-z_k} +  a(z_k)+b(z_k)(z-z_k) +\mathcal{O}(z-z_k)^2.
    \end{equation}
    Substituting (\ref{m2laurent}) into (\ref{m2res}) and (\ref{m2p-2}) respectively, we acquire
    \begin{align}
       & \res_{z=z_k} m^{(2)}(z) =a(z_k) n^{(2)}_{0,k} +b(z_k) n^{(2)}_{1,k},\label{resm2an}\\
       & \underset{z=z_k}{P_{-2}} m^{(2)}(z) =a(z_k)n^{(2)}_{1,k}.\label{p-2m2an}
    \end{align}
    And then, bring (\ref{resm2an}) and (\ref{p-2m2an}) back to (\ref{m2laurent}), we have the Laurent  expansion
    \begin{equation}
        m^{(2)}(z)= a(z_k) \left[ I + \frac{n^{(2)}_{0,k}}{z-z_k}  + \frac{n^{(2)}_{1,k}}{(z-z_k)^2}\right] + b(z_k)\frac{n^{(2)}_{1,k}}{z-z_k} +b(z_k)(z-z_k)+ \mathcal{O}(z-z_k)^2.
    \end{equation}
    Notice that $m^{(2)}$ and $m^{(2)}_{RHP}$ have the same residue relations and $\text{det} m^{(2)}(z)=\text{det} m^{(2)}_{RHP}(z)=1$, it can be calculated directly
    \begin{equation}
        \left(m^{(2)}_{RHP}\right)^{-1}(z)= \left[ I - \frac{n^{(2)}_{0,k}}{z-z_k}  - \frac{n^{(2)}_{1,k}}{(z-z_k)^2}\right]  \sigma_2 a(z_k)^T \sigma_2 + \left(   I- \frac{n^{(2)}_{1,k}}{(z-z_k)^2}      \right) \sigma_2 b(z_k)^T \sigma_2 \left(z-z_k\right) +\mathcal{O}(z-z_k)^2.
    \end{equation}
    Then,
    \begin{equation}
        m^{(2)}(z)\left(m^{(2)}_{RHP}\right)^{-1}\left(z\right)=\mathcal{O}(1),
    \end{equation}
    in which we have used the  property of nilpotent matrix $n^{(2)}_{i,k}$ ($i=0,1$).
    Hence, $m^{(3)}(z)$ has only removable singularities at each $z_k$.

    The property (b) can be obtained by $\bar{\partial}m^{(2)}_{RHP}=0$.
\end{proof}

 The solution of this RHP11 is constructed  by the following integral equation
 \begin{equation}\label{dem}
    m^{(3)}(z)=I-\frac{1}{\pi}\iint_\mathbb{C}\dfrac{m^{(3)}(s)w^{(3)}(s)}{s-z}dA(s),
 \end{equation}
where $dA(s)$ is the Lebesgue measure on  $\mathbb{C}$.
   Meanwhile, the equation (\ref{dem})  can also be represented by operators, which is
  \begin{equation}
    (I-C)m^{(3)}(z)=I,
  \end{equation}
  where $C$ is the Cauchy-Green integral operator,
  \begin{equation}
    C[f](z)=-\frac{1}{\pi}\iint_\mathbb{C}\dfrac{f(s)w^{(3)}(s)}{s-z}dA(s).
  \end{equation}
  In addition, this  operator  $C$ admits the following estimation.
  \begin{proposition}
    For $t\to \infty$,
    \begin{equation}
        \parallel C\parallel_{L^\infty\to L^\infty}\lesssim |t|^{-1/4},
	\end{equation}
	which implies that  $\left(I-C \right) ^{-1}$ exists.
  \end{proposition}

  As $z \to \infty$, we consider the asymptotic expansion of $m^{(3)}(z)$
  \begin{equation}
      m^{(3)}(z) = I + \frac{m^{(3)}_1}{z}+\frac{1}{\pi}\iint_\mathbb{C}\dfrac{s m^{(3)}(s)w^{(3)}(s)}{z(s-z)}dA(s),
  \end{equation}
  where
  \begin{equation}
    {m^{(3)}_1}= \frac{1}{\pi}\iint_\mathbb{C}m^{(3)}(s)w^{(3)}(s)dA(s).
  \end{equation}
 To reconstruct the solution $q(x,t)$ of the fNLS equation with double poles,  we need to determine the long time asymptotic behavior of $m^{(3)}_1$.
   We can testify the following property of $m^{(3)}_1$.
   \begin{proposition}
     There is a constant $c$ such that
       \begin{equation}
        |m^{(3)}_1| \le c t^{-3/4}.
       \end{equation}
   \end{proposition}
   \begin{proof}
       The proofs of Proposition 9 and 10 are similar to the proof of Proposition 6.1 and Appendix D in \cite{fNLS} respectively.
   \end{proof}

\section{Long time asymptotics for fNLS equation}
   Now we begin to seek the long time asymptotic behavior of the solution  for the fNLS equation.

\begin{theorem}\label{thm1}
    Take $q_0(x) \in  H^{1,1}(\mathbb{R})$ and suppose that the corresponding scattering data is $\sigma_d= \left\{ (z_k,c_{i,k}), i=0,1, z_k\in \mathcal{Z}  \right\}_{k=1}^N$ with $z_k$ is double zeros of the scattering coefficient $s_{11}(z)$.
    Fix $x_1, x_2,v_1,v_2 \in \mathbb{R}$ with $x_1 \le x_2, v_1 \le v_2$.
    Take $I=[-v_2/2, -v_1/2]$ and $z_0=-x/(2t)$.
    Take $q_{sol}\left(x,t;\sigma_d^-(I)\right)$ be the $N(I)$ solition corresponding to the scattering data
     $$\sigma_d^-(I)=\left\{ \left(z_k,c_{i,k}(I)\right), i=0,1, z_k\in \Delta_{z_0}^-(I) \right\}_{k=1}^N.$$
    Then as $|t| \to \infty$ with $(x,t)\in \mathcal{S}(x_1,x_2,v_1,v_2)$, we have
    \begin{equation}\label{finalq}
         q(x,t)=q_{sol}\left(x,t;\sigma_d^-(I)\right) + t^{-1/2}f +\mathcal{O}\left(t^{-3/4}\right),
    \end{equation}
    where
    \begin{equation}\label{f}
         f = \left(\eta_{11}\right)^2 \alpha(z_0) e^{i\left(x^2/(2t)-\nu(z_0) \rm{log}|4t|\right)}+(\eta_{12})^2 \overline{ \alpha(z_0)} e^{-i\left(x^2/(2t)-\nu(z_0) \text{\rm{log}} |4t|\right)},
    \end{equation}
    with $| \alpha(z_0)|^2=|v(z_0)|$, $v(z_0)=-\frac{1}{2 \pi} \text{\rm{log}} (1+|r(z_0)|^2)$ and
    \begin{equation}
        \text{\rm{arg}}  \alpha(z_0)=\frac{\pi}{4} + \text{\rm{arg}} \Gamma\left( i\nu(z_0) \right)-   \text{\rm{arg}} \gamma(z_0) - 4 \sum_{k \in \Delta^-_{z_0}} \text{\rm{arg}}(z_0-z_k)- 2 \int_{-\infty}^{z_0} \text{\rm{ln}} |s-z_0|d \text{\rm{ln}} \left(1+|r(s)|^2\right),
    \end{equation}
    where $ \eta_{11}, \eta_{12}$ is the elements in the first row of $m^{\triangle^-_{z_0}(I)}\left(z;y,t|\sigma^-_d(I)\right)$.

\end{theorem}

\begin{proof}
    Reviewing a series of transformations we have made in the process of solving the initial problem (\ref{fNLS})-(\ref{fNLSid}), which are (\ref{transm1}), (\ref{transm2}), (\ref{m2rhp}) and (\ref{transm3}), and backward pushing these transformation processes gives us
    \begin{equation}
       m(z)=m^{(3)}(z)E(z)m^{(out)}\left(\mathcal{R}^{(2)}(z)\right)^{-1}T^{\sigma_3}, \quad z \in \mathbb{C} \setminus \epsilon_{z_0}.
        \end{equation}
    In particular, we consider  cases where $z$ tends to infinity in the vertical direction of $z\in \Omega_2$ or $\Omega_5$. In these cases,  we have $\mathcal{R}^{(2)}=I$ and
    \begin{equation}
        m(z)=\left(I + \frac{m^{(3)}_1}{z}   +\mathcal{O}\left(\frac{1}{z^2}\right)              \right) \left(I + \frac{E_1}{z}   +\mathcal{O}\left(\frac{1}{z^2}\right)                \right) \left(I + \frac{m^{(out)}_1}{z}   +\mathcal{O}\left(\frac{1}{z^2}\right)              \right) \left(I + \frac{T_1 \sigma_3}{z}   +\mathcal{O}(\frac{1}{z^2})              \right).
    \end{equation}
    After that, we get
    \begin{equation}
        m_1=m_1^{(out)}+E_1 +m_1^{(3)}+ T_1\sigma_3,
    \end{equation}
    where $m_1$ is the coefficient of the $z^{-1}$ in the Laurent  expansion of $m$.
    Meanwhile, the equation (\ref{qzm12}) and Proposition 10 tell us that
    \begin{equation}\label{qm1E1}
        q(x,t)=2i \left( \left(m^{(out)}_1\right)_{12} +\left(E_1\right)_{12}   \right) +\mathcal{O}\left(t^{-3/4}\right).
    \end{equation}

     Let $m^{(out)}= \left( \begin{array}{ll}
        \eta_{11} & \eta_{12}\\
        \eta_{21} & \eta_{22}
     \end{array} \right).$
     Using Proposition 8, we have
    \begin{equation}\label{E112}
         (E_1)_{12}=\frac{1}{2i\sqrt{t}} \left(\beta_{12} (\eta_{11})^2+\beta_{21} (\eta_{12})^2 \right),
    \end{equation}
    where
    \begin{equation}
        \beta_{12}(z_0)=\overline{\beta_{21}(z_0)}= \alpha(z_0)e^{i \left(  x^2/(2t)-\nu(z_0)\rm{log}|4t|        \right)}.
    \end{equation}
    Bringing $2i \left(m^{(out)}_1\right)_{12}=q_{sol}\left(x,t;\sigma_d^{(out)}\right)$ and (\ref{E112}) back to (\ref{qm1E1}), we find
    \begin{equation}
        q(x,t)=q_{sol}\left(x,t;\sigma_d^{(out)}\right) + t^{-1/2}f +\mathcal{O}\left(t^{-3/4}\right),
    \end{equation}
    where $f$ is given by (\ref{f}).
    In addition, using Corollary 2, we obtain
    \begin{equation}
        q(x,t)=q_{sol}\left(x,t;\sigma_d^-(I)\right) + t^{-1/2}f +\mathcal{O}\left(t^{-3/4}\right),
    \end{equation}
    where $(x,t)\in \mathcal{S}(x_1,x_2,v_1,v_2)$.
\end{proof}

\begin{remark}
    Though the asymptotic result (\ref{finalq}) has the same form with that
    obtained in \cite{fNLS}, they have different meanings. For example, the first term
     $q_{sol}\left(x,t;\sigma_d^-(I)\right)$ demonstrates high-order pole solutions, while it denotes simple  pole solutions in
      \cite{fNLS};  The second  term $t^{-1/2}f$ is an interaction between high-order pole solutions and
       the dispersion term, while it denotes an interaction between simple pole solutions and
       the dispersion term in  \cite{fNLS}.
\end{remark}

\begin{remark}
    The asymptotic result (\ref{finalq}) shows that the initial value problem of the fNLS equation with zero-boundary and double poles in scattering coefficient has the property of the soliton resolution, which is
    as $t\to \infty$, any solution of the fNLS equation can be decomposed into solitary wave part and dispersion part.
    Linear NLS equation $i q_t+q_{xx}/2$ is  dispersive and any solution of this linear equation has the estimation $\parallel q\parallel_{L^\infty} \sim t^{-1/2}$.
    Therefore, the second term in the formula (\ref{finalq}), which includes the $t^{-1/2}$, represents the contribution of the dispersion term.
    The multiple solitary wave solutions $q_{sol}\left(x,t;\sigma_d^-(I)\right)$ corresponding to the scattering data,  which is superposed by a finite single soliton solution, appear in the long time asymptotic expansion when the NLS equation includes a nonlinear term $|q|^2q$.
\end{remark}

\begin{remark}
   After modification on   the residue conditions (\ref{res3.3})-(\ref{res3.6}), we can show  that the solutions of the
    Cauchy problem of the fNLS equation with high-order pole spectrum data still possess  the property of soliton resolution
    like  Theorem \ref{thm1}.
\end{remark}

\noindent\textbf{Acknowledgements}

This work is supported by  the National Natural Science
Foundation of China (Grant No. 11671095,  51879045).

\end{document}